\documentclass[12pt]{article}
\usepackage{amsthm}
\usepackage{amsmath}
\usepackage{cite}
\usepackage{tikz}
\usetikzlibrary{arrows}
\usepackage{amssymb}
\usepackage{enumerate}
\usepackage[margin=1.3in]{geometry}

\usepackage{float}

\newtheorem{theorem}{Theorem}[section]

\newtheorem{proposition}[theorem]{Proposition}

\newtheorem{lemma}[theorem]{Lemma}
\newtheorem{definition}[theorem]{Definition}
\newtheorem{example}[theorem]{Example}

\newtheorem{mainthm}{Theorem}

\include{epsf}
\makeatletter
\newcommand*{\rom}[1]{\expandafter\@slowromancap\romannumeral #1@}
\makeatother

\begin{document}

\title{Prime Fano $4$-folds with semi-free torus actions}

\author{Nicholas Lindsay}

\maketitle
\begin{abstract} 

Let $X$ be a smooth complex prime Fano fourfold having a semi-free action of $\mathbb{C}^*$, then $X$ is contained in one of the families $\mathbb{P}^4,Q^4,W_5, X^{m}_{8}$. All of the families contain members that have a semi-free $\mathbb{C}^*$-action.

\end{abstract}

\tableofcontents
\section{Introduction}
The classification of smooth complex Fano varieties of a given dimension $n$ is a well-known problem in algebraic geometry. The classification is complete in dimensions at most $3$ \cite{I1,I2,MM}. A natural follow-up question is: which smooth complex Fano varieties can admit an effective action of an infinite algebraic group? This is also complete in dimension at most $3$ \cite{CPS,KPS}. 

In dimension $4$, the number of known smooth complex prime Fano varieties admitting torus actions is the same as in dimension $3$; there are $4$ examples. In addition to the examples $\mathbb{P}^4,Q^4$, there are the exceptional examples $W_5,V_{18}$ (see \cite{PZ3}), both admitting an effective action of a two dimensional algebraic torus $\mathbb{C}^* \times \mathbb{C}^*$. Many fundamental results about the equivariant geometry of $W_5,V_{18}$ were given in \cite{PZ3}, natural subvarieties which are invariant by the automorphism group are studied.   Prokhorov and Zaidenberg classified the possible automorphism groups of fourfolds in the family $V_{18}$ \cite[Theorem A]{PZ4}, there are four possibilities up to isomorphism: $GL_2 \rtimes \mathbb{Z}_2, (\mathbb{C}^*)^2 \rtimes  \mathbb{Z}_{6}, (\mathbb{C}^*)^2 \rtimes  \mathbb{Z}_{2}$ and $(\mathbb{C} \times \mathbb{C}^*) \rtimes  \mathbb{Z}_2$. Contrastingly, $W_5$ is unique up to isomorphism \cite{Fu}, its automorphism group was describe by Piontkowski and Van de Ven \cite[Theorem 6.6]{PV}, it is a connected algebraic group of dimension $7$ with maximal torus having dimension $2$. 

Some results on smooth complex prime Fano $4$-folds having an action of a $2$-dimensional algebraic torus $\mathbb{C}^* \times \mathbb{C}^*$ with isolated fixed points were obtained in \cite{GLS}. When the number of fixed points is at most $6$; which holds for all known examples, a complete classification was given. This paper is concerned with the classification of smooth complex prime Fano fourfolds having a semi-free action of $\mathbb{C}^*$. An action of a group $G$ is called semi-free if the stabilizer of every point is either $G$ or $\{1_G\}$. The first main result of the paper is as follows, it is proved in Theorem \ref{fanomain}.

\begin{mainthm}  \label{thm:mainB} 
Let $X$ be a smooth prime Fano $4$-fold having a semi-free $\mathbb{C}^*$-action, then $X$ is contained in one of the families $\mathbb{P}^4,Q^4,W_5$ or $X^{m}_{8}$. 
\end{mainthm}
Here $Q^4$ is a quadric, $W_5$ is the del Pezzo fourfold of degree $5$ and $X^{m}_{8}$ is the Mukai fourfold of genus $8$ and degree $14$. All four of the families contain fourfolds that admit a semi-free action of $\mathbb{C}^*$. $\mathbb{P}^4$ has two semi-free $\mathbb{C}^*$-actions, described in Example \ref{projectiveone} and Example \ref{exampletwofour}, $Q^4$ has two semi-free $\mathbb{C}^*$-actions, described in Example \ref{quadricone} and Example \ref{quadricexample}, and $W_5$ has a semi-free $\mathbb{C}^*$-action described in Example \ref{Wexample}. $X^{m}_{8}$ contains fourfolds having a semi-free $\mathbb{C}^*$-action, described in Example \ref{kuznetsov}. 

Via the Białynicki-Birula decomposition theorem \cite{BB}, there is also a connection to a question of Hirzebruch \cite[Problem 27]{Hi} about classifying compactifications of $\mathbb{C}^n$ with second Betti number equal to $1$, an early paper in the study of this question is \cite{V}. 

 For a projective manifold $X$ with a torus action  and a fixed component $F \subset X^{\mathbb{C}^*}$, let $N_{F}$ denote the normal bundle of $F$ in $X$. Projective manifolds $X_{1},X_{2}$ each having $\mathbb{C}^*$-actions are called FP-equivalent (fixed point equivalent), if there exists an isomorphism $\varphi: X_{1}^{\mathbb{C}^*} \rightarrow X_{2}^{\mathbb{C}^*} $ such that the weights of the torus action at $p$ are equal to those at $\varphi(p)$ and for each fixed component $F \subset X_{2}^{\mathbb{C}^*}$ and integer $i$, $\varphi^{*}(c_{i}(N_{F})) = c_{i}(N_{\varphi^{-1}(F)})$.  Using this notion, a more precise version of Theorem \ref{thm:mainB} is given, it is proved in Theorem \ref{fpequivalent}.

\begin{mainthm}  \label{thm:mainC} 

Let $X$ be a smooth complex prime Fano $4$-fold with a semi-free $\mathbb{C}^*$-action. Then, one of the following cases occurs: \begin{itemize}
\item[a).] $X \cong \mathbb{P}^4$ and the action is FP-equivalent to the one of the actions given in Example \ref{projectiveone} and Example \ref{exampletwofour}.
\item[b).] $X \cong Q^4$ and the action is FP-equivalent to one of the ones given in Example \ref{quadricone} and Example \ref{quadricexample}.
\item[c).] $X \cong W_5$ and the action is FP-equivalent to the action described in Example  \ref{Wexample}.

\item[d).] $X$ is contained in the family $ X^{m}_{8}$. $X$ is FP-equivalent to the action described in Example \ref{kuznetsov}.
\end{itemize}

\end{mainthm}

The proof uses the theory of Hamiltonian $S^1$-actions. A prime Fano fourfold $X$ with a algebraic $\mathbb{C}^*$-action admits an equivariant polarization, and thus an equivariant embedding into complex projective space with respect to some linear action. For a linear $\mathbb{C}^*$-action on projective space, the action of the compact group $S^1$ is well known to be Hamiltonian, therefore pulling back the Hamiltonian gives $X$ a Hamiltonian $S^1$-action with K\"{a}hler form $\omega$ normalized so that $c_{1}(X) = [\omega]$ (see Lemma \ref{algebraicrestriction} for a more detailed argument).

The majority; but not all, of the steps to proving Theorem \ref{thm:mainB} apply equally closed  simply connected $8$-dimensional symplectic manifolds with $b_{2}(M)=1$ and a semi-free Hamiltonian $S^1$-action. The following theorem summarizes the results which are obtained in a purely symplectic context. It is proved in Theorem \ref{symplecticmain} with the exception of statement d). which is proved in Theorem \ref{bfourbound}. In this setting the index $\iota(M)$ is defined in direct analogy with the Fano index; the maximal integer $k$ for which $\frac{c_1(M)}{k}$ is an integral cohomology class.

\begin{mainthm}  \label{thm:mainA} 
Let $(M,\omega)$ be a closed simply connected symplectic $8$-manifold with $b_{2}(M)=1$ with a semi-free Hamiltonian $S^1$-action. Then the following hold:

\begin{itemize}
\item[a).] $M$ has Todd genus $1$.
\item[b).] All non-isolated fixed  components are symplectomorphic to either: $$ \mathbb{CP}^1,\;\mathbb{CP}^2, \; \mathbb{CP}^1 \times \mathbb{CP}^1,\; \mathbb{CP}^3,$$ with the (product of) Fubini-Study symplectic forms.

\item[c).] $1 < \iota(M)$.

\item[d).] $b_{4}(M) \leq 14$.
\end{itemize}
\end{mainthm}

The assumption that $b_{2}(M)=1$ forces $(M,\omega)$ to be positive monotone (see Lemma \ref{symfano}), hence these results fit into the problem of classifying positive monotone symplectic manifolds with Hamiltonian torus actions, on which there has been recent work by several authors, see \cite{Ch,LP1,SS,CK3}. Recall that the subsets of the manifold where the Hamiltonian attains its minimum and maximum are connected symplectic submanifolds denoted $M_{\min},M_{\max}$. To prove Theorem \ref{thm:mainA}, we analyse cases depending on the value of $(d_{1},d_{2}): = (\dim(M_{\min}),\dim(M_{\max}))$. Reversing the circle action i.e. replacing $X$ with $-X$ and $H$ with $-H$ which has the effect of swapping the minimum and maximum submanifolds. Therefore, we can assume that $d_1 \leq d_2 < 8.$ Since $M_{\min},M_{\max}$ are symplectic and therefore even dimensional this gives ten possible cases.

The proof relies upon an array of results of the theory of Hamiltonian $S^1$-actions. For a brief summary let us mention, Kirwan localization, Atiyah-Bott-Berline-Vergne localization and the Duistermaat-Heckman theorem. A more surprising ingredient is a result of Jannich and Osa about the signature of the fixed point set of an involution on a manifold, in Proposition \ref{existfour} this is used in combination with  \cite[Theorem 3.4]{L} to prove the existence of a fixed component of dimension at least $4$, which is a starting point for further analysis in Section \ref{excases}, and in particular rules out five of the ten possibilities for $(d_1,d_2)$.

In Section \ref{excases} the main tools are the ABBV localization formula and Proposition \ref{spheremaps} which gives natural spheres in the reduced spaces, under different conditions on $(d_1,d_2)$ and moreover gives information about the integral of the reduced symplectic form on them via the Duistermaat-Heckman theorem. With the exception of the case $(d_1,d_2) =(0,0)$ which uses an entirely different method, the ABBV localization formula for the equivariant cohomology class $1$ and Proposition \ref{spheremaps} determine the fixed components up to symplectomorphism and their normal bundles, up to the number of isolated fixed points with weights $\{-1,-1,1,1\}$. The number of such fixed points is bounded in Section \ref{DHA}.

Let us also mention the proof in the exceptional case  $(d_1,d_2) =(0,0)$. By Proposition \ref{existfour} there is an non-extremal fixed component of dimension $4$, by analysis how that the reduced symplectic form changes closed to this fixed submanifold, we show it has degree two in the reduced space and by \cite{H} that it is symplectomorphic to $\mathbb{CP}^1 \times \mathbb{CP}^1$ with the product of Fubini-Study symplectic forms.   

In the last Section \ref{Autfan}, Theorem \ref{thm:mainA} is used to study semi-free algebraic torus actions on smooth prime Fano fourfolds. The key consequence of Theorem \ref{thm:mainA}  is that the Fano index is greater than $1$, which allows us to appeal to the classification of prime Fano fourfolds of index greater one from \cite{Mu,W}. We narrow the list down to eight possibilities in Lemma \ref{fanofourfoldposdef} using the fact that the intersection form is positive definite \cite[Theorem 3.4]{L}. Then, by combining Theorem \ref{thm:mainA} with some brief additional arguments involving the Duistermaat-Heckman function, Theorem  \ref{thm:mainB} and Theorem  \ref{thm:mainC} are proven. 

\textbf{Acknowledgements}. I would like to thank Alexander Kuznetsov for providing Example \ref{kuznetsov}, answering a question posed in the first version of this paper. The author is funded by the position: Postdoc in complex and symplectic geometry in memory of Paolo de Bartolomeis. 
\section{Preliminaries}
\subsection{Hamiltonian $S^1$-actions}
Let $(M,\omega)$ be a closed symplectic manifold, suppose that there is smooth $S^1$-action on $M$ generated by a vector field $X$.  A Hamiltonian $S^1$-action on  $(M,\omega)$ is defined as follows.

\begin{definition}
The $S^1$-action is called Hamiltonian if $$dH = -\omega(X,\cdot)$$ for some smooth function $H:M \rightarrow \mathbb{R}$.
\end{definition}
The function $H$ is called the Hamiltonian of the action, the equation $dH = -\omega(X,\cdot)$ is called the Hamiltonian equation. For a subgroup $G \subset S^1$ define a subset of $M$ as follows: $$M^{G} := \{p \in M \mid z.p = p \;\; \forall z \in G\}.$$ 

A connected component of $M^{S^1}$ is called a fixed component, a connected component of $M^{\mathbb{Z}_m}$, $m \geq 2$ is called an isotropy submanifold. Because $X$ generates the $S^1$-action, $M^{S^1}$ is equal to the zero set of the $X$;  $\mathcal{Z}(X)$, which by the Hamiltonian is equal to zero set of $dH$, $\mathcal{Z}(dH)$ i.e. the critical set of $H$:
$$\mathcal{Z}(X) = \mathcal{Z}(dH) = M^{S^1}.$$

There is an $S^1$-invariant almost complex structure $J$ compatible with $\omega$ \cite[Lemma 5.52]{MS}. It follows that for any $p \in M^{S^1}$, $T_{p}(M)$ inherits the structure of a complex $S^1$-representation, there are $n$ weights $w_{i} \in \mathbb{Z}$ which encode the decomposition into irreducibles: $$z.(z_1,\ldots,z_n) = (z^{w_1}z_1, \ldots, z^{w_1}z_1). $$ A key fact is that the number of negative weights at a fixed point $p$ counted with multiplicity is equal to half the Morse Bott index of $H$ at $p$, for a proof see \cite[Lemma 5.54]{MS}.

The following result of Kirwan \cite{Ki} will be central to the proofs in the paper. It allows the Betti numbers of $M$ to be calculated as a sum of the Betti numbers of the fixed components. 

\begin{theorem}  \cite{Ki} \label{homologyloc}
Let $(M,\omega)$ be a closed symplectic manifold having a Hamiltonian $S^1$-action. Then for all $i$, $$b_{i}(M) = \sum_{F \subset M^{S^1}} b_{i-2\lambda_{F}}(F). $$
\end{theorem}

Another result which will be a central part of our tool-kit will be the Duistemaat-Heckman theorem. To state this, we have to recall that if $(M,\omega)$ is closed symplectic manifold with a Hamiltonian $S^1$-action and $c$ is a regular value of $H$, then the quotient space $$M_{c} = H^{-1}(c)/S^1,$$ has the structure of a symplectic orbifold \cite[Lemma 5.35]{MS}, which is referred to in various different ways in the literature, as the Marsden-Weinstein quotient, the symplectic quotient or the reduced space. We will use the term reduced space. Attached to each reduced space the level set $H^{-1}(c)$ has the structure of an orbifold principal $S^1$-bundle over $M_{c}.$ If $J$ is an invariant almost complex structure, and $g$ the associated invariant Riemannian metric then the gradient vector field of $H$ with respect to $g$ is $Y : =JX$, the flow of $Y$ gives equivariant diffeomorphisms between level sets $H^{-1}(c_1),H^{-1}(c_2)$ provided $H^{-1}(c_1,c_2) \cap M^{S^1} = \emptyset$, and therefore orbifold diffeomorphisms of the reduced spaces. It is via these gradient flow maps that variation of the cohomology class of the symplectic form is interpreted in the Duistermaat-Heckman theorem. We are now ready to state the Duistermaat-Heckman theorem in its most basic form:

\begin{theorem} \label{dhthe} \cite[Theorem 1.1]{DH}
Suppose that $(M,\omega)$ is a closed symplectic manifold with a Hamiltonian $S^1$-action, with Hamiltonian $H$. Let $c$ be a regular value of $H$. Then $$\frac{d}{dt}[\omega_t] |_{t=c}= e(H^{-1}(c)) .$$ 

Here $H^{-1}(c)$ is considered as an (orbifold) principal $S^1$-bundle over the reduced symplectic orbifold $M_{c}$ and $e$ denotes the Euler class of this bundle in $H^{2}(M_{t},\mathbb{R})$.
\end{theorem}

The Duistermaat-Heckman theorem also has many other forms, mostly related to expressing the symplectic volume of $M$ in terms of data associated to the moment map of the action. Define $$DH(x) = \int_{M_{x}} \frac{[\omega_x]^{n-1}}{(n-1)!} ,$$ it is well-defined on the set of regular values of $H$ and is piecewise polynomial and continuous. Then another form of the Duistermaat-Heckman formula states that \begin{equation}\label{dhformula} \int_{M} \frac{\omega^n}{n!} = \int_{H(M)} DH(x). \end{equation}

An action of a group $G$ is called semi-free if the stabilizer of every point is either $G$ or $\{1_G\}$. For a semi-free Hamiltonian $S^1$-actions on a closed symplectic manifold, the  objects in the above theorem take a simpler form. In this case, for each regular value $c$ of $H$, $M_{c}$ is a closed symplectic manifold, $H^{-1}(c)$ is an principal $S^1$-bundle over $M_{c}$ in the sense of manifolds. 

The following result about closed symplectic manifolds with $b_{2}(M)=1$ and a non-trivial Hamiltonian $S^1$-action will be needed. 

\begin{lemma}  \label{symfano}  Suppose that $(M,\omega)$ is a closed symplectic manifold with a non-trivial Hamiltonian $S^1$-action and $b_{2}(M)=1$. Then after possibly rescaling the symplectic form by a positive constant it holds that that $$c_{1}(M)=[\omega].$$
\end{lemma}
\begin{proof}
Since $b_{2}(M)=1$, $c_{1}(M) = k [\omega]$ for some $k \in \mathbb{R}$. The result then follows from \cite[Proposition 1.1]{On}.\end{proof}

The following formula is well-known \cite{CK,GVS,F}, for a proof in the stated generality see for example \cite[Proposition 1.9]{LP1}.
\begin{theorem} \label{wsf}
Suppose that $(M,\omega)$ is a closed symplectic manifold with $c_{1}(M) = [\omega]$ with a Hamiltonian $S^1$-action with Hamiltonian $H: M \rightarrow \mathbb{R}$. Then, after possibly adding a constant to the Hamiltonian, for each fixed component $F$, $$ H(F) = -\sum_{i=1}^{n} w_i ,$$ where $w_i$ are the weights along $F$.
\end{theorem}

The following lemma, organizes the link between $\mathbb{C}^*$ actions on smooth complex Fano varieties and Hamiltonian $S^1$-actions on them.

\begin{lemma} \label{algebraicrestriction}
Let $X$ be a smooth complex Fano variety and with a $\mathbb{C}^*$-action. Then, \begin{itemize} \item[a).] The restriction of the action to $S^1 \subset \mathbb{C}^*$ is Hamiltonian with respect to a K\"{a}hler form $\omega$ satisfying $c_{1}(X)= [\omega]$.
\item[b).] The action of $S^1 \subset \mathbb{C}^*$ is semi-free if and only if the  $\mathbb{C}^*$-action is semi-free.
\end{itemize}
\end{lemma}
\begin{proof}
Recall the natural group isomorphism  $\mathbb{C}^* \cong S^1 \times \mathbb{R}_{>0}$. There exists a $\mathbb{C}^*$-equivariant embedding of $X$ into some projective space, with respect a some linear $\mathbb{C}^*$-action on projective space $\mathbb{P}^{N}$ associated to a sufficiently large power of the anticanonical line bundle \cite[Lemma 5.3]{GLS}. A linear action on $\mathbb{C}^*$-action may be conjugated by a unitary projective transformation $\varphi \in PU(N)$ to a diagonal action: 

$$z.[z_0 :  z_1 : \ldots : z_N] =[z_0 : z^{w_{1}} z_1 : \ldots : z^{w_{N}} z_N],$$ for some $w_1,\ldots,w_N \in \mathbb{Z}$. 

The action of $S^1 \subset \mathbb{C}^*$ with generating vector field $X_{0}$ is Hamiltonian with respect to the Fubini-Study form, with Hamiltonian $$H ([z_0 :  z_1 : \ldots : z_N]) = \frac{\sum_{i=1}^{n} \pi w_i |z_i|^2}{\sum_{i=0}^{n} |z_i|^2},$$ that is $dH = -\omega(X_0,\cdot)$, proving a). The gradient vector field of $H$ with respect to the Fubini-Study metric of $JX_0$, where $J$ is the standard complex structure $\mathbb{P}^N$. Since the action of $\mathbb{C}^*$ is holomorphic,  that action of $\{1\} \times \mathbb{R}_{>0} \subset \mathbb{C}^*$ is defined by the flow is $JX_0$. To prove b)., first suppose that the action of $S^1$ is semi-free.  It follows the action $\mathbb{R}_{>0}$ is semi-free, since it is generated by the gradient vector field of a function on a closed manifold. So via the isomorphism $\mathbb{C}^* \cong S^1 \times \mathbb{R}_{>0}$ it follows that the action of $\mathbb{C}^*$ is semi-free. Note also that this also implies $X^{S^1} = X^{\mathbb{C}^*}$, since the vector field generating the action of  $\{1\} \times \mathbb{R}_{>0}$;  $JX_0$ and the vector field generating the circle action; $X_0$ have the same zero set. Then, if the action of $\mathbb{C}^*$ is semi-free then using the condition $X^{S^1} = X^{\mathbb{C}^*}$ the semi-free condition for the subgroup $S^1$-action follows. \end{proof}

\subsection{Formulas coming from the ABBV localization class of the normal bundle of a $4$-dimensional fixed component.}

The following is the well-known equivariant splitting principle for equivariant complex vector bundles.
\begin{lemma} \label{eqsplit}
(Equivariant Splitting Principle) Let $X$ be a paracompact topological space with an $S^1$-action, and $E$ an $S^1$-equivariant vector bundle over $X$. Then there exists a topological space $Y$ equipped with an $S^1$-action and an $S^1$-equivariant map $p: Y \rightarrow X$ such that: 
\begin{itemize}
\item $p^{*}E$ decomposes equivariantly as a sum of $S^1$-equivairant line bundles.
\item $p^*: H^{*}_{S^1}(X) \rightarrow H^{*}_{S^1}(Y)$ is injective.
\end{itemize}
\end{lemma}

\begin{lemma} \label{abbvfourformula}
Let $M$ be a closed symplectic $8$-manifold with a Hamiltonian $S^1$-action, and $F$ a $4$-dimensional fixed component Let $N_F$ denote the normal bundle of $F$. Then, if the non-zero weights both equal to $1$, $$e^{S^1}(N_{F}) = (t^2 + c_{1}(N_{F})t + c_{2}(N_{F}) ). $$ Similarly, if the non-zero weights both equal to $-1$, $$e^{S^1}(N_{F}) = (t^2 - c_{1}(N_{F})t + c_{2}(N_{F}) ). $$
\end{lemma}
\begin{proof}

We just prove the first case, the proof of the second follows by the same argument. Write $E=N_{F}$. Let $p: Y \rightarrow X$ be the $S^1$-equivariant map provided by Lemma \ref{eqsplit} so that $p^{*}E$ splits equivariantly as a sum of equivariant line bundles $p^{*}E \cong L_1 \oplus L_2$. Note that the map on equivariant cohomology induced by $p$ preserves the equivariant generator i.e. the pull-back of the hyperplane class from $\mathbb{CP}^{\infty}$ in the Borel model. This may be seen directly via the Borel model construction. In a small abuse of notation we denote this class by $t$ in the equivariant cohomology ring of $X$ and $Y$.  

 Since the fibers of $p^*(E)$ are isomorphic as $S^1$-representations to those of $E$, the weights of the action on $L_i$ are both $1$. Therefore $c^{S^1}(L_i) = 1 + t + c_1(L_{i})$ for $i=1,2$.  Therefore by the Whitney sum formula $$c_{2}^{S^1}(L_1 \oplus L_2) = t^2 + t(c_1(L_{1})+ c_1(L_{2})) + c_1(L_{1})c_1(L_{2}) =  t^2 + t(c_1(L_{1} \oplus L_{2})) + c_2(L_{1} \oplus L_{2}),$$ substituting via the equation $p^*E \cong L_1 \oplus L_2$ gives $$ c_{2}^{S^1}(p^*(E)) = ( t^2 + tc_1(p^*(E)) + c_2(p^*(E))).$$ Combining the fact that $p^*$ is injective and $p^*(t)=t$ noted above, by linearity of the induced homomorphism and functoriality of Chern classes via pull-back this yields $$ c_{2}^{S^1}(E) =  t^2 + tc_1(E) + c_2(E),$$ which in turn yields the desired formula when one notes that since $N_{F}=E$ has rank $2$, by definition $ e^{S^1}(N_{F}) = c_{2}^{S^1}(N_{F})$. \end{proof}

The following formula follows from the ABBV localization formula. The expressions for the individual terms for $\dim(F) \leq 2$ were given in \cite[Remark 2.5]{T}.
\begin{proposition} \label{abbvone}Let $(M,\omega)$ be a closed symplectic manifold with dimension $8$ having a semi-free Hamiltonian $S^1$-action. Then, $$\sum_{ F \subset M^{S^1} }\frac{1}{e^{S^1}(N_{F})} = 0 .$$ Moreover, the contribution of fixed submanifolds may be computed as follows. \begin{enumerate}\item If $p$ is an isolated fixed point, then $$ \frac{1}{e^{S^1}(N_{F})}  = (-1)^{\lambda(p_{i})} .$$\item  if $\Sigma$ is a fixed surface  let $a_{k}$ ($k=1,\ldots, 3$ be the degrees of the components of the normal bundle of $\Sigma_{j}$. We label the components so that $a_{k}$ correspond to components with weight $-1$ for $1\leq k \leq \lambda(\Sigma)$ and weight $1$ for $\lambda(\Sigma) +1\leq k \leq 3$, then $$ \frac{1}{e^{S^1}(N_{F})}  =  (-1)^{\lambda(\Sigma) +1} ( - a_{1} -  \ldots - a_{ \lambda(\Sigma)} +  a_{\lambda(\Sigma) +1} + \ldots +a_{3} ).  $$

\item If $N$ is an extremal fixed component with dimension $4$ and $b_{2}(F)=1$, then $$\frac{1}{e^{S^1}(N_{F})} = c_{1}(N_{F})^2 -c_{2}(N_{F}).$$ 
\end{enumerate}
\end{proposition}
\begin{proof}
This follows from integrating the equivariant cohomology class $1$ via the ABBV localization formula. The expressions for the terms corresponding to isolated fixed point and surfaces are given in \cite[Remark 2.5]{T}.

Let $F$ be an extremal four-dimensional fixed component satisfying $b_{2}(M)=1$, we assume first that $F = M_{\min}$. Let $u$ be the integral generator of $H^{2}(F, \mathbb{Z})$ then $e^{S^1}(N_F) = t^2 + c_{1} u t + c_{2} u^2$, where $c_{1},c_{2} \in \mathbb{Z}$ which represent $c_{i}(N_{F})$ in the above basis of the cohomology of $F$. If $F = M_{\min}$ then by Lemma \ref{abbvfourformula}: $$ \frac{1}{ t^2 + c_{1} u t + c_{2} u^2} =  \frac{t^{-2}}{ 1 + c_{1} u t^{-1} + c_{2} u^2t^{-2} }  =   \frac{t^{-2} (1 - c_{1} u t^{-1})}{ 1 - c_{1}^2 u^2 t^{-2} + c_{2} u^2t^{-2} }  $$

$$= \frac{t^{-2} (1 - c_{1} u t^{-1})(1 + c_{1}^2 u^2 t^{-2} - c_{2} u^2t^{-2})}{ 1}  = c_{1}^2 -c_2.$$
 If $F = M_{\max}$ then by Lemma \ref{abbvfourformula}:  $$ \frac{1}{ t^2 - c_{1} u t + c_{2} u^2} =  \frac{t^{-2}}{ 1 - c_{1} u t^{-1} + c_{2} u^2t^{-2} }  =   \frac{t^{-2} (1 + c_{1} u t^{-1})}{ 1 - c_{1}^2 u^2 t^{-2} + c_{2} u^2t^{-2} }  $$

$$= \frac{t^{-2} (1 + c_{1} u t^{-1})(1 + c_{1}^2 u^2 t^{-2} - c_{2} u^2t^{-2})}{ 1}  = c_{1}^2 -c_2.$$
\end{proof}

\subsection{$4$-dimensional positive monotone symplectic manifolds}
The last preliminary result regards closed symplectic $4$-manifolds $N$ which satisfy $[\omega] = k c_{1}(N)$ for some $k > 0$. let $\omega_{FS}$ denote the Fubini-Study symplectic form on complex projective space.

\begin{theorem} \label{projectiveplanetheorem} \cite{OO}
Let $(N,\omega)$ be a closed symplectic $4$-manifold such that $[\omega] = k c_{1}(N)$ for some $k > 0$. Then, up to scaling the symplectic form by an appropriate positive constant, $(N,\omega)$ is symplectomorphic to a del Pezzo surface, with the pull-back of $\omega_{FS}$ via a projective embedding associated to a sufficiently high power of the anticanonical bundle.
\end{theorem}
\begin{proof} By \cite{OO}, the manifold is diffeomorphic to a del Pezzo surface, by \cite[Conclusion. Page 10]{S}, this gives the claim about the symplectomorphism type.
\end{proof}

 \section{Ruling out  $(d_1,d_2) = (2,6),(4,6), (6,6),(2,2)(0,2)$. } \label{rulingout}
In this part of the article, we consider closed symplectic $8$-manifolds with $b_{2}(M)=1$ and a semi-free Hamiltonian $S^1$-action.

 Recall that the subsets of the manifold where the Hamiltonian attains its minimum and maximum are connected symplectic submanifolds denoted $M_{\min},M_{\max}$. To prove the main classification results of this paper, the argument works in cases of the value: $(d_{1},d_{2}) = (\dim(M_{\min}),\dim(M_{\max}))$. Reversing the circle action i.e. replacing $X$ with $-X$ and $H$ with $-H$ which has the effect of swapping the minimum and maximum submanifolds. Therefore, we can assume that $$d_1 \leq d_2 < 8.$$ Since $M_{\min},M_{\max}$ are symplectic and therefore even dimensional this gives ten possible cases, in this section we discard five of the cases, all the other five are furnished by examples and are further analysed in Section \ref{excases}.

The following Lemma rules out three of the cases, in fact here the semi-free assumption is not needed, but later in the section the arguments rely on this condition crucially.

\begin{lemma} \label{ruleoutone} There does not exist a closed symplectic $8$-manifold having $b_{2}(M) =1$ and a Hamiltonian $S^1$-action such that $(d_1,d_2) = (2,6),(4,6), (6,6)$. 
\end{lemma}
\begin{proof}
Suppose that there is such an action satisfying $(d_1,d_2) = (2,6),(4,6), (6,6)$. Recall that $\lambda(M_{\min})=0$ and $\lambda(M_{\max}) = 8-d_2 = 2$. Theorem \ref{homologyloc}  gives $$ b_{2}(M) =  \sum_{F \subset M^{S^1}} b_{2-2\lambda_{F}}(F) = b_{2}(M_{\min}) + b_{0}(M_{\max}) + c, $$ where $c$ is a non-negative integer. Which contradicts the assumption $b_{2}(M)=1$.
\end{proof}

\begin{proposition} \label{existfour}
Let $(M,\omega)$ be a closed symplectic $8$-manifold having a semi-free Hamiltonian $S^1$-action and $b_{2}(M)=1$ then $M^{S^1}$ has a connected component of dimension at least $4$.
\end{proposition}
\begin{proof}
Assume for a contradiction that every component of $M^{S^1}$ has dimension at most $2$. By the semi-free assumption $M^{\mathbb{Z}_2} = M^{S^1}$. By \cite[Theorem 3.4]{L2} and the assumption that all fixed components have dimension $0$ and $2$, $$\sigma(M) =  b_{4}(M) > 0.$$ By \cite[Main Result]{JO} and since the action of $- 1 \in S^1$ is homotopic to the identity: $$\sigma(M) = \sigma(M^{\mathbb{Z}_2}M^{\mathbb{Z}_2}) ,$$ since all components of $M^{\mathbb{Z}_2}$ have dimension less than half the manifold,  this intersection product is $0$, a contradiction. Hence there exists some component $F \subset M^{S^1}$ with $\dim(F) \geq 4$.
\end{proof}

\begin{lemma} \label{ruleouttwo}
There does not exist a closed symplectic $8$-manifold having $b_{2}(M) =1$ and a semi-free Hamiltonian $S^1$-action such that $(d_1,d_2) = (0,2), (2,2)$. 
\end{lemma}
\begin{proof}
Since by assumption $\dim(M_{\max})=2$, $\lambda(M_{\max}) = 3$. By Lemma \ref{existfour} there is a component $F_0 \subset M^{S^1}$ with $\dim(F_0) \geq 4 $ and by the assumptions and Poincar\'{e} duality $b_{0}(M)= b_{2}(M)= b_{6}(M)= b_{8}(M)=1$. By Proposition \ref{homologyloc} the following formula holds:

\begin{equation} \label{ruleoutequation}  
1= b_{8}(M) = b_{2}(M_{\max}) +  \sum_{F \subset M^{S^1}, F \neq M_{\max}} b_{8-2\lambda_{F}}(F).  \end{equation}

Since $F_0$ is non-extremal Equation (\ref{ruleoutequation}) implies $0<\lambda_{F_0} <4$ and since $\dim(F_0) \geq 4$ and $F$ is a symplectic submanifold $b_{2i}(F_0)>0$ for all $i=0, \ldots, \frac{\dim(F_0)}{2}$. Combining this with the Equation  (\ref{oneonetwo}) , the only possibility is $\dim(F)=4$ and $\lambda_{F}=1$.

 Then by Theorem \ref{homologyloc}, $$b_{6}(M)  = \sum_{F \subset M^{S^1}} b_{6-2\lambda_{F}}(F) = b_{0}(M_{\max}) + b_{4}(F_0) + c, $$ where $c$ is a non-negative integer. That is the desired contradiction. 
\end{proof}

\section{Classifying the existent cases} \label{excases}
In this section, we continue to use the notation $(d_1,d_{2}) = (\dim(M_{\min}),\dim(M_{\max}))$ and make the assumption that $d_{1} \leq d_{2}$. Due to Lemma \ref{ruleoutone} and Lemma \ref{ruleouttwo}, it remains to deal with the cases: $$(d_{1},d_{2}) = (0,0),(0,4),(2,4),(4,4),(0,6).$$ There are examples of semi-free Hamiltonian $S^1$-actions on closed symplectic $8$-manifolds with $b_{2}(M)=1$ furnishing all of these cases. The problem then becomes about classification. The following proposition will be central to this analysis, to state the proposition, define $H_{\max} := H(M_{\max})$ and $H_{\min} := H(M_{\min})$.

\begin{proposition} \label{spheremaps}
Let $(M,\omega)$ closed symplectic $8$-manifold having $b_{2}(M) =1$ and a semi-free Hamiltonian $S^1$-action. Scale the symplectic form via Lemma \ref{symfano} so that $c_1 = [\omega]$ and normalize the Hamiltonian $H$  via Theorem \ref{wsf} so that $H(p) = -\sum w_i$.  Then the following statements hold:
\begin{itemize}
\item[a).] If $d_2=4$ and $H^{-1}(0,2) \cap M^{S^1} = \emptyset$ and there is an isolated fixed point $p$ with weights $\{-1,-1,1,1\}$. Then, there exists a smooth map $f: S^2 \rightarrow M_{\max}$ such that $\int_{S^2} f^{*}\omega = 2$.

Similarly, if $H^{-1}(H_{\min},0) \cap M^{S^1} = \emptyset$ and there is an isolated fixed point $p$ with weights $\{-1,-1,1,1\}$. Then, there exists a smooth map $f: S^2 \rightarrow M_{\min}$ such that $\int_{S^2} f^{*}\omega = |H_{\min}|$.
\item[b).]  If $d_2=4$ and every non-extremal fixed component $F$ is isolated and satisfies $\lambda_{F}=2$. Then, there exists a smooth map $f: S^2 \rightarrow M_{\min}$ such that $\int_{S^2} f^{*}\omega = H_{\max} - H_{\min}$.

\item[c).] If $ d_2 \leq 4$ and there are no non-extremal fixed components. Then, there are maps $f_{\min} : S^2 \rightarrow M_{\min}$,   $f_{\max}: S^2 \rightarrow M_{\max}$ such that $\int_{S^2} f_{\min}^{*}\omega = H_{\max} - H_{\min}$ and $\int_{S^2} f_{\max}^{*}\omega = H_{\max} - H_{\min}$.

\item[d).] If $d_1=0$ and there is a fixed surface $F$ with weights $\{0,-1,1,1\}$ and moreover $H^{-1}(-4,1) \cap M^{S^1} = \emptyset$. Then, setting $a_1,a_2,a_3$ to be the degrees of the splitting of the normal bundle, with $a_1$ corresponding to the component with weight $-1$. Then $$3a_1 = 2 + a_1  + a_2 + a_3 .$$
\end{itemize}
\end{proposition}
\begin{proof}
Firstly by the equivariant symplectic neighbourhood theorem  \cite[Theorem A.1]{Ka} each fixed component admits an $S^1$-invariant neighborhood which is equivariantly symplectomorphic to a neighbourhood of the zero section in the normal bundle with a fiberwise symplectic form. Then, picking an invariant compatible almost complex structure and extend it (possibly non-equivariantly) to all of $(M,\omega)$, the averaging procedure of \cite{MS} will give an invariant compatible almost complex structure which agree with the original ones on possibly smaller neighborhoods of the fixed point sets. All the statements are proved using the gradient flow of the corresponding invariant metric $g$ such that $\omega = g(J\cdot, \cdot)$, that is the flow of the gradient vector field $\nabla_g H$of $H$ with respect to $g$, in fact $\nabla_{g}H = JX$ where $X$ is the generating vector field of the $S^1$-action. Recall that a vector field on a closed manifold gives an $\mathbb{R}$-action, a gradient flow line is an orbit of this action containing more than one point. 

To prove $a).$ isolated fixed point $p$ with weights $\{-1,-1,1,1\}$, then note that the normal bundle of $p$ is $T_{p}M$ which splits equivariantly into two rank $2$ subspaces $N_{+},N_{-}$ on which $S^1$ acts with weight $1$ and $-1$. Let $C_+$ denote the union of all gradient flow lines which correspond to $N_{+}$ via the linear identification with the normal bundle constructed at the start of the proof. Let $\pi_c: H^{-1}(c) \rightarrow M_{c}$ denote the usual projection to the symplectic quotient, then note that by the Duistermaat-Heckman theorem Theorem \ref{dhthe}, for $c \in (0,2)$ $\pi_c(C)$ is a $2$-sphere such that $\int_{\pi(C)} \omega = c$.  One the other hand consider a neighborhood $U$ of $M_{\max}$ which is identified with a neighbourhood of the normal bundle with linear symplectic form. Note that $\omega$ induces a $c^\infty$ closed $2$-form on the $5$-manifold with boundary $U/S^1$. Let $f$ the map given by the projection to the base of $\pi_{M_{\min}} : \pi(C) \rightarrow M_{\max}$ for any $c$ sufficiently close to $2$ so that $M_{c} \subset U/S^1$. Using the above properties one verifies that the smooth map $f$ satisfies $\int_{S^2} f^{*}\omega = 2$. The second statement holds by the same argument.

To prove $b)$, let $f_i: S^2 \rightarrow M_{\max}$ be the maps constructed in the previous paragraph, since note that since each $f_i$ is smooth $M_{\max} \setminus (\cup_{i} f_i(S^2))$ is non-empty, take a point $q \in M_{\max} \setminus (\cup_{i} f_i(S^2))$. Then considering the union $C_q$ of the flow lines with maximum $q$, by the way $q$ was chosen,  $C_{q}$ is invariant and the action on it is free, by the Duistermaat-Heckman theorem 
Theorem \ref{dhthe} $$\int_{\pi_{2-c}(C_{q})} \omega = c,$$ then constructing a $2$-sphere as in $M_{\min}$ by the projection to the base in a linear identification with the normal bundle, proves the existence of $f$.

To prove c). Let $q \in M_{\max}$ be arbitrary and let $U_{q}$ be the upward flowing normal bundle, by the assumption $d_2 \leq 4$, for $c$ sufficiently close to $H_{\max} =2$, $\pi_{c}(U_q)$ is symplectomorphic to $(\mathbb{CP}^{\frac{8-d_2}{2} -1},\omega_{FS})$ and $H^{-1}(c)$ is the Hopf bundle over $\pi_{c}(U_q)$, picking any line $l$ in the projective space, then $\int_{l} e(H^{-1}(c)) = 1$. Then, letting $U'$ be the union of gradient flow lines that intersect $\pi_{c}^{-1}(l)$, gives a family of $2$-spheres in the reduced space as in the previous parts of the proof, and the proof then proceeds as above. 

Finally to prove $d).$  Let $U_{F}$ be the union of gradient flow lines with maximum $F$. Then, by the equivariant symplectic neighborhood theorem for $c<-1$ sufficiently close to $-1$, $\int_{\pi_{c}(U_{F})} e(H^{-1}(c)) = a_{1}$, moreover the limit of $\int_{\pi_{c}(U_{F})} \omega$ is zero as $c$ tends to $-4$. Therefore, since $\int_{F} [\omega] = 2 + a_1 + a_2 + a_3$, the result follows. \end{proof}

\subsection{ $(d_1,d_2) = (0,0)$}
Now we begin to classify closed symplectic $8$-manifolds having $b_{2}(M)=1$ and a semi-free Hamiltonian $S^1$-action in the case $(d_1,d_2)=(0,0)$. There is one example in this case which is now described.

The image of the Pl\"{u}cker embedding $Gr(2,4) \rightarrow \mathbb{P}(\wedge^2(\mathbb{C}^4)) =  \mathbb{P}^5$ is a smooth quadric, denoted $Q^4$. In the following example we study a torus action on $Q^4$ constructed via the above identification.
 
\begin{example} \label{quadricone}
There is a $\mathbb{C}^*$-action on $Q^4 \cong Gr(2,4)$ induced by the linear action on $\mathbb{C}^4$ defined by $$z.(z_1,z_2,z_3,z_4) = (zz_1,zz_2,z_3,z_4).$$ Moreover, $$(Q^4)^{\mathbb{C}^4} = \{[P_1],[P_{2}]\} \cup \Pi.$$ Where: $P_1=\{(z_1,z_2,0,0) \mid z_1,z_2 \in \mathbb{C}\}$, $P_2=\{(0,0,z_3,z_4) \mid z_3,z_3 \in \mathbb{C}\}$ and $\Pi := \{[\langle v_1,v_2 \rangle] \mid v_1 \in P_1,  v_{2} \in P_{2}\} \subset Gr(2,4)$.  It further holds that $\Pi  \cong \mathbb{P}^1 \times \mathbb{P}^1$ and the $\mathbb{C}^*$-action is semi-free. 
 \end{example}

The claims about the fixed point set are immediate after unpacking the definitions. To prove that the action is semi-free, it is useful to identify $Gr(2,4)$ with the space $\mathbb{G}(1,3)$ of projective lines in $\mathbb{P}^3$. The associated action on $\mathbb{CP}^3$ is semi-free and has fixed point set $L_{1} \cup L_{2}$, where $L_{1} := \{[z_0:z_1:0:0] \mid [z_0:z_1] \in \mathbb{CP}^1 \}$,  $L_{2} := \{[0:0:z_3:z_4] \mid [z_3:z_4] \in \mathbb{CP}^2 \}$. The fixed point set of the action on $\mathbb{G}(1,3)$  contains $[L_{1}],[L_{2}]$ and the component $\Pi$ corresponds to classes of the lines which intersects both $L_{1}$ and $L_{2}$ in a single point (hence why it is naturally isomorphic to $L_{1} \times L_{2}$). By Lemma  \ref{algebraicrestriction} the action of $\mathbb{C}^*$ is semi-free if its restriction to $S^1$ is. Supposing the action of $S^1$ is not semi-free and $L$ is not fixed by the action on  $\mathbb{G}(1,3)$ then there exists a non-trivial element $g \in S^1$ of finite order such that $g$ preserves $L$, therefore the action restricted to $L$ has at least two fixed points. But since the fixed point set of the associated action on $\mathbb{P}^3$ is $L_{1} \cup L_{2}$, $L$ intersects $L_{1} \cup L_{2}$ in at least two points, implying that $L$ is contained in the fixed point set of the action on $\mathbb{G}(1,3)$, hence the action is semi-free.

Lemma  \ref{algebraicrestriction} the restriction of the $\mathbb{C}^*$-action to $S^1$ is Hamiltonian for a K\"{a}hler form $\omega$ such that $c_1 = [\omega]$. By Theorem \ref{wsf} the Hamiltonian $H$ may be normalized so that $H(p) =- \sum w_i$ for all $p \in M^{S^1}$. Then, by Theorem \ref{homologyloc} $\{M_{\min},M_{\max}\} = \{[L_1],[L_2]\}$ so  $H(M_{\min}) = -4$, $H(M_{\max}) = 4$ and $\Pi$ is a non-extremal fixed component isomorphic to $\mathbb{P}^1 \times \mathbb{P}^1$ with $\lambda_{\Pi} = -1$ and weights $\{-1,1,0,0\}$, and $H(\Pi)=0$.

In the following proposition, it is shown that in the case $(d_1,d_2)  =(2,4)$ the symplectomorphism type of the fixed components and the Chern classes of their normal bundles correspond precisely to Example \ref{quadricone}.

\begin{proposition} \label{oneonecase} Let $(M,\omega)$ be a closed symplectic $8$-manifold with $b_{2}(M)=1$ having a semi-free Hamiltonian $S^1$-action such that $(d_1,d_2) = (0,0)$. Then, there is only one non-extremal fixed component $F$ symplectomorphic to $ \mathbb{CP}^1 \times \mathbb{CP}^1$, with symplectic form $4\omega_{FS} \times4\omega_{FS}$, and normal bundle $N_{F} \cong \mathcal{O}(1,1) \oplus \mathcal{O}(1,1).$
\end{proposition}
\begin{proof}
By Lemma \ref{symfano} and Theorem \ref{wsf} the symplectic form and Hamiltonian may be normalized so that $c_1 = [\omega]$ and $H(p) = -\sum w_i$, where $w_i$ are the weights at $p$. By Lemma \ref{existfour} there is a component $F_0 \subset M^{S^1}$ with $\dim(F_0) \geq 4 $ and by the assumptions and Poincar\'{e} duality $b_{0}(M)= b_{2}(M)= b_{6}(M)= b_{8}(M)=1$. By Proposition \ref{homologyloc} the following formula holds:

\begin{equation} \label{oneonetwo}  
1= b_{8}(M) = b_{0}(M_{\max}) +  \sum_{F \subset M^{S^1}, F \neq M_{\max}} b_{8-2\lambda_{F}}(F).  \end{equation}

Since $F_0$ is non-extremal Equation (\ref{oneonetwo}) implies $0<\lambda_{F_0} <4$ and since $\dim(F_0) \geq 4$ and $F$ is a symplectic submanifold $b_{2i}(F_0)>0$ for all $i=0, \ldots, \frac{\dim(F_0)}{2}$. Combining this with the Equation  (\ref{oneonetwo}), the only possibility is $\dim(F)=4$ and $\lambda_{F}=1$.

Therefore by Proposition \ref{homologyloc} the following formulas hold: 

\begin{equation} \label{oneonethree}  
1= b_{2}(M) =  b_{0}(F_0) + \sum_{F \subset M^{S^1}, F \neq F_0} b_{2-2\lambda_{F}}(F), \end{equation}

\begin{equation} \label{oneonefour}  
1= b_{6}(M) =  b_{4}(F_0)  + \sum_{F \subset M^{S^1} , F \neq F_0} b_{6-2\lambda_{F}}(F).  \end{equation}

For a non-extremal fixed component $F$ such that $F \neq F_{0}$ by Equation (\ref{oneonethree}) and Equation (\ref{oneonefour}) $F$ must be an isolated fixed point $p$ with $\lambda_{p}=2$. The existence of such fixed points is ruled out by \cite[Lemma 3.1]{T}. Therefore $M^{S^1} = M_{\min}\cup M_{\max} \cup  F_0$.

 By \cite[Theorem 10.1-10.2]{GS} $M_0 : = H^{-1}(0)/S^1$ has the structure of a symplectic manifold, we denote this symplectic form $\omega_0$ and by the Duistermaat-Heckman theorem and the Moser stability theorem, it is symplectomorphic to $\mathbb{CP}^3$ with the symplectic form normalized so that $c_{1} = [\omega_0] = 4h$, where $h$ is the hyperplane class. Let $L_1,L_{2}$ be the components of the normal bundle corresponding to the weights $-1,1$. In the proof \cite[Theorem 3.4]{L2} it was shown that $(1-\frac{1}{|H_{\min}|} - \frac{1}{|H_{\max}|} )[\omega]|_{F_0} = c_{1}(F_0)$ and $|H_{\min}|c_1(L_1) = |H_{\max}|c_1(L_2) = [\omega]|_{F_0}$. Since the action is semi-free the weights are $M_{\min},M_{\max}$ are $\{1,1,1,1\}$, $\{-1,-1,-1,-1\}$ respectively and so $|H_{\max}| = |H_{\min}|=4$, therefore $[\omega]|_{F_0} = 2c_{1}(F_0)$. Therefore by \ref{projectiveplanetheorem} $F$ is symplectomorphic to a del Pezzo surface and therefore has Todd genus $1$, that is $\int_{F_0} c_{1}(F_0)^2 + \chi_{top}(F_0) = 12$, rearranging gives $\int c_{1}(F_0)^{2} = 10 - b_{2}(F_0)$.

By the above, it follows that $[\omega_{0}] = 4h$ where $h \in H^{2}(M_0,\mathbb{Z})$, is the hyperplane class, therefore $$  \int_{F_0} [\omega]^2 = \int_{M_{0}} PD(F_0, M_{0}). (4h)^2 = 16k,$$ where $k \in \mathbb{Z}_{>0}$ is the degree of $F_0$ in $M_{0} \cong \mathbb{CP}^3$. So, substituting $[\omega_0]|_{F_0} = 2c_{1}(F_0)$ gives $$\int_{F_0} [\omega]^2 = \int_{F_0} (2 c_{1}(F_0))^2 = 4 (10 - b_{2}(F_0)) \leq 36. $$

 Therefore since $0 < b_{2}(F_0) $, $0<16k \leq 36$, so $k=1,2$. If $k=1$ then $F_0$ is symplectomorphic to $\mathbb{CP}^2$ but by the above formula for $c_{1}(F_0)^2$, $b_{2}(F_0)=6$ which is a contradiction. Therefore $k=2$ by \cite[Theorem 1.3]{H}, it is symplectomorphic to $\mathbb{CP}^1 \times \mathbb{CP}^1$, with symplectic form $a \pi_1^{*}\omega_{FS} +  b \pi_2^{*}\omega_{FS}$. Using bidegree notation to express classes in $H^{2}(\mathbb{CP}^1 \times \mathbb{CP}^1,\mathbb{Z}) \cong \mathbb{Z}^2$, $c_1 = (2,2)$ and so the monotonicity implies symplectic form satisfies $[\omega_{F_0}] =  (4,4)$. By the above equations $|H_{\min}|c_1(L_1) = |H_{\max}|c_1(L_2) = [\omega]|_{F_0}$ it follows that $c_{1}(L_{1}) = c_{1}(L_{2}) = (1,1)$. \end{proof}

\subsection{ $(d_1,d_2) = (0,4)$}

In this subsection, we focus on the case $(d_1,d_2) = (0,4)$. In this case the primary example is an action on $W_5$, the del Pezzo $4$-fold of degree $5$.
\begin{example} \label{Wexample}
Let $W_5$ be the quintic del Pezzo fourfold. Let $\mathbb{C}^* \times \mathbb{C}^*$ act on $W_5$ via \cite[Theorem 5.2]{GLS} (see also \cite{PZ1}). Then the $\mathbb{C}^*$-action of the linear subtorus $z \mapsto (z,z^{-1})$ is semi-free and $W_{5}^{\mathbb{C}^*} = \{p,\Sigma,\Pi \}$,  where $p$ is an isolated fixed point,  $\Sigma \cong \mathbb{CP}^1$ and $\Pi \cong \mathbb{CP}^2$. 
\end{example}
By the expression of the weights given in \cite[Theorem 5.2]{GLS}, the weights at $p$ are $\{1,1,1,1\}$, the weights at $\Sigma$ are $\{0,-1,1,1\}$ and the weights at $\Pi $ are $\{0,0,-1,-1\}$. By Lemma  \ref{algebraicrestriction} the restriction of the $\mathbb{C}^*$-action to $S^1$ is Hamiltonian for a K\"{a}hler form $\omega$ such that $c_1 = [\omega]$. Letting $M=W_5$, by Theorem \ref{wsf} the Hamiltonian $H$ may be normalized so that $H(p) =- \sum w_i$ for all $p \in M^{S^1}$.  Therefore $H(p)=-4$, $H(\Sigma) = -1$, and $H(\Pi) = 2$, in particular $M_{\min} = \{p\}$, $M_{\max}= \Pi$. Moreover, if $N_{\min/\max}$ denotes the normal bundle of $M_{\min/\max}$ and $\int_{\Sigma} c_{1}(M)=9$, $c_{1}(N_{\max})=0$ and $\int_{M_{\max}} c_{2}(N_{\max})=2$, where $N_{\max}$ denotes the normal bundle of $M_{\max}$ \cite[Lemma 3.6]{PZ1}.

The case $(d_1,d_2) = (0,4)$ differs from the others in the sense that there are two possibilities for the dimension of non-extremal fixed components ($0$ and $2$). The arguments with localization formulas take a different form when there exists a fixed component of dimension $2$ or not, firstly we consider the case that there is no fixed component of dimension $2$. 

\begin{lemma} \label{zerofourcaseone}
Let $(M,\omega)$  closed symplectic $8$-manifold with $b_{2}(M)=1$ with a semi-free Hamiltonian $S^1$-action such that $(d_1,d_2)=(0,4)$ and $M^{S^1}$ consists of $M_{\max}$ and only isolated fixed points. Then if $N_{\max}$ denotes the normal bundle of $M_{\max}$, then: $M_{\max}$ is symplectomorphic to $\mathbb{CP}^2$, $$\int_{M_{\max}}c_{2}(N_{\max}) = b_{4}(F),$$ $$  \int_{M_{\max}} c_{1}^2(N_{\max})  = 1. $$ If $b_{4}(M)>1$ then $c_{1}(N_{\max})=-L$, where $L$ represents the Poincar\'{e} dual of a line. 
\end{lemma}
\begin{proof}

By Lemma \ref{symfano} the symplectic form may be multiplied by a positive constant so that $c_1 = [\omega]$. Let $H$ be the Hamiltonian produced by Theorem \ref{wsf}. Since $\dim(M_{\min}) = 0$ and the action is semi-free the weights at $M_{\min}$ are $\{1,1,1,1\}$ and $H(M_{\min}) = -4$ and the action is semi-free the weights at $M_{\max}$ are $\{-1,-1,0,0\}$  and $H(M_{\max}) = 2$. By the assumptions and Poincar\'{e} duality $b_{2}(M)=b_{6}(M)=1$. By Proposition \ref{homologyloc}:
\begin{equation} \label{zerofourloctwo}  
1= b_{6}(M) = b_{2}(M_{\max}) +  \sum_{F \subset M^{S^1}, F \neq M_{\max}} b_{6-2\lambda_{F}}(F).  \end{equation}

Since $M_{\max}$ is a symplectic $4$-manifold, $b_{2}(M_{\max})=1$. Let $N_{i}$ denote the number of fixed points $p$ with $\lambda_{p} = N_{i}$, applying Proposition \ref{homologyloc}:  $1=b_{0}(M)=N_{0}$, $1 = b_{2}(M)= N_{1}$. Similarly, $N_{3}=0$ and $N_{4}=0$ and finally the equation \begin{equation} \label{zerofourlocthree}  
 b_{4}(M) = 1 + N_{2}  \end{equation} holds. By the semi-free condition, every fixed point $p$ with $\lambda_{p} = 2$ has weights $\{-1,-1,1,1\}$

 If $N_{\max}$ denotes the normal bundle of  $M_{\max}$ then by Proposition \ref{abbvone} and Equation (\ref{zerofourlocthree}):  \begin{equation}\label{zerofourloc} \int_{M_{\max}} c_{1}^2(N_{\max}) - \int_{M_{\max}} c_{2}(N_{\max}) + b_{4}(M) -1 = 0. \end{equation}

On the other hand by \cite[Theorem 3.4]{L2} the intersection form of $M$ is positive definite, i.e. $\sigma(M) = b_{4}(M)$. By \cite[Main Result]{JO}, $$M^{\mathbb{Z}_2}.M^{\mathbb{Z}_2} = \sigma(M) = b_{4}(M). $$ Furthermore,  since the action is semi-free $M^{\mathbb{Z}_2} =M^{S^1}$ consists of $M_{\max}$ and a finite collection of isolated fixed points, $$M^{\mathbb{Z}_2}.M^{\mathbb{Z}_2} = \int_{M_{\max}} c_{2}(N_{\max}) = b_{4}(M).$$  Therefore, there is a cancellation in Equation (\ref{zerofourloc}) simplifying the equation to:  \begin{equation}\label{zerofourloctwo}   \int_{M_{\max}} c_{1}^2(N_{\max})  = 1.  \end{equation}

By Proposition \ref{homologyloc} they have $b_{2}(M_{\max})=1$ and since $M_{\max}$ is a symplectic $4$-manifold it satisfies $\sigma(F)=1$, therefore by the Hirzebruch signature theorem $\int_{M_{\max}} c_{1}(M_{\max})^2 = 9.$ Let $a$ be the integral generator of the second homology which is a positive multiple of $[\omega]|_{M_{\max}}$. Then $c_{1}(M_{\max}) = \pm 3a$. In fact, $c_{1}(M_{\max})=-3a$ is impossible, because by the equation $\int_{M_{\max}} c_{1}(N_{M_{\max}})^2 =1$ shown in Equation (\ref{zerofourloctwo}), then $[\omega]|_{M_{\max}} = (-3 + k')a$ where $k'=\pm1$ which contradicts the definition of $a$. Therefore $c_{1}(M_{\max})=3a$, therefore $[\omega]|_{M_{\max}} = (3 + k')a$, where $c_{1}(N_{\max})=k'a$. In particular, $M_{\max}$ is positive monotone, and therefore since $b_{2}(F)=1$, by Theorem \ref{projectiveplanetheorem} $F$ is symplectomorphic to $\mathbb{CP}^2$. Since $b_{4}(M)>1$ then by Equation (\ref{zerofourlocthree}) $N_{2}>0$ there must exist an fixed point with weights $\{-1,-1,1,1\}$ therefore by Proposition \ref{spheremaps} there exists a map  $f: S^2 \rightarrow M_{\max}$, such that $\int_{S^2} (f^{*}(\omega) ) = 2$. Therefore it must fold that $(3+k')$ is at most $2$, showing that $k'=-1$ as required.
\end{proof}
The following proposition treats the case when $M^{S^1}$ does contain a component of dimension $2$, as is the case for the action on $W_{5}$ described in Example \ref{Wexample}.

\begin{proposition} \label{zerofourcasetwo}
Let $M$ be a closed symplectic $8$-manifold with $b_{2}(M)=1$ with a semi-Hamiltonian $S^1$-action  such that $(d_1,d_2)=(0,4)$ and such that $M^{S^1} \setminus M_{\max}$ has a component of positive dimension. Then there exists a fixed component $\Sigma$ of dimension $2$ with $\lambda_{F}=1$. Moreover, and, if $N$ denotes the normal bundle of $M_{\max}$ $\int_{M_{\max}} c_{1}(N_{\max})^2 \in \{0,1\}$, moreover $M_{\max}$ is symplectomorphic to $\mathbb{CP}^2$. Moreover, $$\int_{M_{\max}}c_{2}(N_{\max}) = b_{4}(F)$$ and  $$\int_{\Sigma} c_{1}(M)  = 3(3- \int_{M_{\max}} c_{1}(N_{\max})^2).$$If additionally  $b_{4}(M) > 2$, then $c_{1}(N_{\max})=-L$ where $L$ is the Poincar\'{e} dual of a line in particular $\int_{M_{\max}} c_{1}(N_{\max})^2 =1$ in this case.
\end{proposition}
\begin{proof}
By Lemma \ref{symfano} the symplectic form may be multiplied by a positive constant so that $c_1 = [\omega]$. Let $H$ be the Hamiltonian produced by Theorem \ref{wsf}. Since $\dim(M_{\min}) = 0$ and the action is semi-free the weights at $M_{\min}$ are $\{1,1,1,1\}$ and $H(M_{\min}) = -4$ and the action is semi-free the weights at $M_{\max}$ are $\{-1,-1,0,0\}$  and $H(M_{\max}) = 2$. 

\begin{equation} \label{zerofoursloctwo}  
1= b_{6}(M) = b_{2}(M_{\max}) +  \sum_{F \subset M^{S^1}, F \neq M_{\max}} b_{6-2\lambda_{F}}(F).  \end{equation}

Since $M_{\max}$ is a symplectic $4$-manifold, $b_{2}(M_{\max})=1$.  Let $F$ be a component in $M^{S^1} \setminus F$ with positive dimension, since $M_{\min}$ is isolated $\dim(F)>0$, moreover by Equation (\ref{zerofoursloctwo})  $b_{6-2\lambda_{F}}(F)=0$ since $F$ is a connected symplectic submanifold of positive dimension $b_{2}(F)>0$ $b_{0}(F)=1$  therefore $\lambda_{F}=1$ and again by Equation  (\ref{zerofoursloctwo}) $\dim(F)=2$, define $\Sigma := F$. Note that $b_{0}(\Sigma) = b_{2}(\Sigma) =1$, also since the action is semi free the weights at $F$ are $\{0,-1,1,1\}$ and $H(F)=-1$.

\begin{equation} \label{zerofourslocthree}  
1= b_{2}(M) = b_{0}(\Sigma) +  \sum_{F \subset M^{S^1}, F \neq \Sigma} b_{2-2\lambda_{F}}(F).  \end{equation}

Therefore by Equations (\ref{zerofoursloctwo}), (\ref{zerofourslocthree}) every fixed component not equal to $M_{\min}$,$M_{\max}$, $\Sigma$ satisfies $\dim(F) = 0 $ and $\lambda_{F} = 2$, therefore has weights $\{-1,-1,1,1\}$. Let $N_{2}$ denote the number of such isolated fixed points. Then by Proposition \ref{homologyloc}:

\begin{equation} \label{zerofourslocfour}  
b_{4}(M) =  \sum_{F \subset M^{S^1}, F \neq \Sigma} b_{4-2\lambda_{F}}(F) = b_{0}(M_{\max}) + b_{2}(\Sigma) + N_{4} = 2 + N_{2}.  \end{equation}

Then, denote by $N_{\Sigma}$ the normal bundle of $\Sigma$ in $M$, and label the component of the splitting of $N_{\Sigma}$ so that $a_1$ is the first Chern class of the component of weight $-1$ and $a_2,a_3$ are the first Chern class of the components with weight $1$. Then if $N_{\max}$ denotes the normal bundle of $M_{\max}$  by Proposition \ref{abbvone} \footnote{Note that while it was shown that there are precisely $b_{4}(M)-2$ fixed points $p$ with $\lambda_p=2$, in this case the isolated fixed point $M_{\min}$ also contributes $+1$ to the formula of Proposition \ref{abbvone}, hence the given form.}:  \begin{equation}\label{zerofoursurloc} \int_{M_{\max}} c_{1}^2(N_{\max}) - \int_{M_{\max}} c_{2}(N_{\max}) -a_1 + a_2 +a_3 + b_4(M) - 1 = 0. \end{equation}

On the other hand by \cite[Theorem 3.4]{L2} the intersection form of $M$ is positive definite, i.e. $\sigma(M) = b_{4}(M)$. By \cite[Main Result]{JO}, $$M^{\mathbb{Z}_2}.M^{\mathbb{Z}_2} = \sigma(M) = b_{4}(M). $$ Furthermore,  since the action is semi-free $M^{\mathbb{Z}_2} =M^{S^1}$ consists of $M_{\max}$ and a finite collection fixed submanifolds of dimension at most two, $$M^{\mathbb{Z}_2}.M^{\mathbb{Z}_2} = \int_{M_{\max}} c_{2}(N_{\max}) = b_{4}(M).$$  Therefore, there is a cancellation in Equation (\ref{zerofourloc}) simplifying the equation to:  \begin{equation}\label{zerofoursurloctwo}   \int_{M_{\max}} c_{1}^2(N_{\max})  -a_1 + a_2 + a_3 = 1.  \end{equation} Since $H(\Sigma)=-1$ and $H(M_{\min}) =-4$ by Proposition \ref{spheremaps} $2+ a_1 + a_2 + a_3 = 3a_1$, therefore $-2 + a_1 = -a_1 + a_2 + a_3 $. Substituting this into Equation \ref{zerofoursurloctwo}  implies that \begin{equation} \label{coneoversigma} \int_{M_{\max}} c_{1}^2(N_{\max}) + a_1  = 3. \end{equation} Since $\int_{M_{\max}} c_{1}(N_{\max})^2$ is a square number, and $a_1>0$ (as the symplective volume of $F$ is $3a_1$), the only possibilities are $\int_{M_{\max}} c_{1}(N_{\max})^2 =1$, $a_{1}=2$,  $\int_{M_{\max}} c_{1}(N_{\max})^2 =0$, $a_{1}=3$. The claim about the integral of $c_{1}(M)$ over $\Sigma$ follows from Equation (\ref{coneoversigma}).

By Equation (\ref{zerofoursloctwo}) $b_{2}(M_{\max})=1$ and since $M_{\max}$ is a symplectic $4$-manifold it satisfies $\sigma(F)=1$, therefore by the Hirzebruch signature theorem $\int_{M_{\max}} c_{1}(M_{\max})^2 = 9.$ Let $a$ be the integral generator of the second homology which is a positive multiple of $[\omega]|_{M_{\max}}$. Then $c_{1}(M_{\max}) = \pm 3a$. In fact, $c_{1}(M_{\max})=-3a$ is impossible, because by the equation $\int_{M_{\max}} c_{1}(N_{M_{\max}})^2 =0,1$ shown above, then $[\omega]|_{M_{\max}} = (-3 + k')a$ where $k'=0,\pm1$ which contradicts the definition of $a$. Therefore $c_{1}(M_{\max})=3a$, in particular $M_{\max}$ is positive monotone, and therefore by Theorem \ref{projectiveplanetheorem} $F$ is symplectomorphic to $\mathbb{CP}^2$. Finally, if $b_{4}(M) > 2$, then $c_{1}(N)=-L$, by Proposition \ref{spheremaps}. \end{proof}

\subsection{ $(d_1,d_2) = (0,6)$}
In this subsection the case $(d_1,d_2)=(0,6)$ is analysed, the primary example in this case is an action on $\mathbb{P}^4$ described below.

\begin{example} \label{projectiveone}
Let $\mathbb{C}^*$ act on $\mathbb{P}^4$ by $$z.[z_0:z_1:\ldots:z_4] = [zz_0:zz_1: z z_2:z z_3: z_4] .$$  Then  $(\mathbb{P}^4)^{\mathbb{C}^*} = \{p\} \cup \Xi$ where $p= [0:0:0:0:1]$ and $\Xi = \{[z_0:z_1:\ldots:z_4] \in \mathbb{P}^4 \mid  z_4=0\}$.
\end{example}
 By Lemma  \ref{algebraicrestriction} the restriction of the $\mathbb{C}^*$-action to $S^1$ is Hamiltonian for a K\"{a}hler form $\omega$ such that $c_1 = [\omega]$. Letting $M=\mathbb{P}^4$, by Theorem \ref{wsf} the Hamiltonian $H$ may be normalized so that $H(p) =- \sum w_i$ for all $p \in M^{S^1}$. By direct computation in affine charts the weights at $p$ are $\{1,1,1,1\}$ and the weights at $H$ are $\{-1,0,0,0\}$. Therefore $M_{\min}=\{p\}$, $H(p)=-4$, $M_{\max}=\Xi$, $H(\Xi)=1$. Moreover, if $N_{\max}$ denotes the normal bundle over $M_{\max}$, then $N_{\max} \cong \mathcal{O}(1)$.

\begin{lemma} \label{zerosixcase}
Let $(M,\omega)$ be a closed symplectic $8$-manifold having a semi-free Hamiltonian $S^1$-action and satisfying $b_{2}(M)=1$. Assume furthermore, $(d_1.d_{2}) = (0,6)$. Then, the fixed point set has two components, $M_{\min}$ an isolated fixed point and $M_{\max}$ is symplectomorphic to $\mathbb{CP}^3$, $c_{1}(N_{\max}) = h,$ where $h \in H^{2}(\mathbb{CP}^3,\mathbb{Z})$ is the hyperplane class.
\end{lemma}
\begin{proof}
 By the assumptions and Poincar\'{e} duality $b_{2}(M)=b_{6}(M)=1$. Since $M_{\max}$ is a $6$-dimensional symplectic submanifold, $b_{i}(M_{\max}) >0$ for $i=0,2,4,6$ and $\lambda(M_{\max}) = 1$. Therefore, by Proposition \ref{homologyloc} 

\begin{equation} 
1= b_{6}(M) = b_{4}(M_{\max}) +  \sum_{F \subset M^{S^1}, F \neq M_{\max}} b_{6-2\lambda_{F}}(F).  \end{equation} Therefore $b_{4}(M_{\max})=1$ and therefore by Poincar\'{e} duality $b_{2}(M_{\max})=1$. By Proposition \ref{homologyloc} every non-extremal fixed component is isolated and has $\lambda_{p}=2$. By the semi-free assumption, the only remaining possibility for fixed submanifolds is isolated fixed points with weights $\{-1,-1,1,1\}$, this impossible by \cite[Lemma 3.1]{T}. Therefore, $M$ is a simple Hamiltonian manifold in the sense of \cite[Definition 1.1]{HH}, and by \cite[Corollary 6.5]{HH} the result is proven. \end{proof}

\subsection{ $(d_1,d_2) = (2,4)$}

In this subsection the case $(d_1,d_2)=(2,4)$ is analysed, the primary example in this case is an action on $\mathbb{P}^4$ described below.

\begin{example} \label{exampletwofour}
Let $\mathbb{C}^*$ act on $\mathbb{P}^4$ by $$z.[z_0:z_1: z_2: z_3  :z_4] = [z_0: z_1:zz_2:zz_3:zz_4] .$$ Then $(\mathbb{P}^4)^{\mathbb{C}^*} = L \cup \Pi$, where $L$ is a line and $\Pi$ is a plane.
\end{example}
By Lemma  \ref{algebraicrestriction} the restriction of the $\mathbb{C}^*$-action to $S^1$ is Hamiltonian for a K\"{a}hler form $\omega$ such that $c_1 = [\omega]$. Letting $M=\mathbb{P}^4$, by Theorem \ref{wsf} the Hamiltonian $H$ may be normalized so that $H(p) =- \sum w_i$ for all $p \in M^{S^1}$.  By direct computation is affine charts the weights at $L$ are $\{0,1,1,1\}$ and the weights at $\Pi$ are $\{-1,-1,0,0\}$. Therefore $M_{\min}=L$, $H(L)=-3$, $M_{\max}=\Pi$, $H(\Pi)=2$. Moreover, if $N_{\min/\max}$ denotes the normal bundle over $M_{\min/\max}$, then $N_{\min} \cong  \mathcal{O}(1)^3$ and $N_{\max} \cong \mathcal{O}(1)^2$. So applying the Whitney sum formula the following formulas hold: $c(N_{\min}) = 1 + 3h$, $c(N_{\max}) = 1 + 2h + h^2$, where $h$ is the hyperplane class.

In the following proposition, it is shown that in the case $(d_1,d_2)  =(2,4)$ the symplectomorphism type of the fixed components and the first Chern classes of the normal bundles correspond precisely to Example \ref{exampletwofour}.

\begin{proposition} \label{twofourcase}
Let $(M,\omega)$ be a closed simply connected symplectic $8$-manifold having $b_{2}(M)=1$ and having a semi-free Hamiltonian $S^1$-action. Suppose that $(d_1,d_2) = (2,4)$. Let $N_{\min},N_{\max}$ denote the normal bundle of $M_{\min},M_{\max}$ respectively.  Then, $b_{4}(M)=1$ and $M^{S^1}=M_{\min} \cup M_{\max}$. Furthermore, $M_{\max}$ is symplectomorphic to $\mathbb{CP}^2$, $\int_{M_{\max}} c_{1}(N_{\max})^2 = 4$, $M_{\min}$ is symplectomorphic to $\mathbb{CP}^1$ and $\int_{M_{\min}} c_{1}(N_{\min}) = 3.$ 

\end{proposition}
\begin{proof}

By Lemma \ref{symfano} the symplectic form may be multiplied by a positive constant so that $c_1 = [\omega]$. Let $H$ be the Hamiltonian produced by Theorem \ref{wsf}. Since $\dim(M_{\min}) = 2$ and the action is semi-free the weights at $M_{\min}$ are $\{0,1,1,1\}$ and $H(M_{\min}) = -3$ $\dim(M_{\max}) = 4$ and the action is semi-free the weights at $M_{\max}$ are $\{-1,-1,0,0\}$  and $H(M_{\max}) = 2$. 

By Poincar\'{e} duality and the assumptions $b_{2}(M)=b_{6}(M)=1$. Therefore by Proposition \ref{homologyloc} the following formulas hold: 

\begin{equation} \label{twofourlocone}  
1= b_{2}(M) = b_{2}(M_{\min}) +  \sum_{F \subset M^{S^1}, F \neq M_{\min}} b_{2-2\lambda_{F}}(F).  \end{equation}

Since $M_{\min}$ is a closed orientable $2$-manifold $b_{2}(M_{\min}) = 1$. Similarly

\begin{equation} \label{twofourloctwo}  
1= b_{6}(M) = b_{2}(M_{\max}) +  \sum_{F \subset M^{S^1}, F \neq M_{\max}} b_{6-2\lambda_{F}}(F).  \end{equation}

Since $M_{\max}$ is a symplectic $4$-manifold, $b_{2}(M_{\max})=1$. Therefore Equations (\ref{twofourlocone}),(\ref{twofourloctwo}) imply that for any non-extremal fixed component $F$, $b_{i-2\lambda_{F}}(F) = 0$ for $i=2,6$.  Applying the formulas for $b_{0}(M),b_{8}(M)$ in a similar way show that for any non-extremal fixed component $F$ $b_{i-2\lambda_{F}}(F) = 0$ for $i=0,8$.  The equations  $b_{i-2\lambda_{F}}(F) = 0$ for $i=0,2,6,8$ imply $\lambda_{F}=2$, and that $\dim(F)=0$.   Therefore since the action is semi-free the weights of the action along $F$ are $\{-1,-1,1,1\}$ for any non-extremal fixed component.

By Lemma \ref{abbvone}, $$\int_{M_{\max}} c_{1}(N_{\max})^2 - \int_{M_{\max}} c_{2}(N_{\max}) - \int_{M_{\min}} c_{1}(N_{\min}) + (b_{4}(M)-1)=0 .$$  By \cite[Theorem 3.4]{L2} the intersection form of $M$ is positive definite, i.e. $\sigma(M) = b_{4}(M)$, where $\sigma$ denotes the signature. By \cite[Main Result]{JO}, $$M^{\mathbb{Z}_2}.M^{\mathbb{Z}_2} = \sigma(M) = b_{4}(M) . $$ Since the action is semi-free $M^{\mathbb{Z}_2} = M^{S^1}$ By Proposition \ref{homologyloc} and the assumptions $b_{2}(M)=1$, $(d_1,d_2)=(2,4)$, it follows that $M_{\min}$ is the only component with dimension at least $4$, $\int_{M_{\max}} c_{2}(M_{\max}) =b_{4}(N)$. Therefore, we obtain \begin{equation} \label{eqtwofour} \int_{M_{\max}} c_{1}(N_{\max})^2  - \int_{M_{\min}} c_{1}(N_{\min})  -1 =0. \end{equation}

Now the proof splits into two case, when $b_{4}(M)=1$ or when $b_{4}(M)>1$, note that the second case is equivalent to the existence of isolated fixed points with weights $\{-1,-1,1,1\}$.

If $b_{4}(M)> 1$ then by Proposition \ref{spheremaps}  there is  a map  $f: S^2 \rightarrow M_{\max}$, such that $\int_{S^2} (f^{*}(\omega) ) = 2$.  Let $a$ be the generator of $H^{2}(M_{\max},\mathbb{Z})$ modulo torsion, such that $a = c[\omega]$ for some $c>0$. Then set, $c_{1}(N_{\max}) = k'a$, $k' \in \mathbb{Z}$, note that since $b_{2}(M_{\max})=1$ and $M_{\max}$ is a symplectic $4$-manifold $\sigma(M_{\max})=1$. Moreover since $M$ and hence $F$ \cite{Li} is simply connected $b_{1}(F)=0$  and $\chi_{top}(F)=3$.  and therefore by the Hirzebruch signature theorem $\int_{M_{\max}} c_{1}(M_{\max})^2 = 9$. Therefore $c_{1}(M_{\max}) = -3a$ or $c_{1}(M_{\max}) = 3a$. It will be shown that the first case is impossible, suppose for a contradiction that  $c_{1}(M_{\max}) = -3a$.   Write $PD(f_{*}([S^2])) = k$, then it holds that  $(-3+k') k = 2$, and since $-3 +k'>0$, this implies $-3 +k' = 1,2$, implying $k' = 4,5$, $\int_{M_{\max}} c_{1}(N_{\max})^2 = 16,25$, therefore by  Equation (\ref{eqtwofour}), $\int_{M_{\min}} c_{1}(N_{\min}) = 15,24$ and therefore $c_{1}(M)|_{M_{\min}}$ is either $17$ or $26$ times the generator of $H^{2}(M_{\min}) \cong \mathbb{Z}$, which contradicts the existence of a map   $f: S^2 \rightarrow M_{\min}$, such that $\int_{S^2} (f^{*}(\omega) ) = 3 .$ 

Therefore $c_{1}(M_{\max} ) =3a$, the equation becomes $(3+k')k=2$ with  $3+k'>0$ therefore $(3+k') =1,2$ and correspondingly $c_{1}(N_{\max}) = -2a$ or  $c_{1}(N_{\max}) = -a$. In the case that $c_{1}(N_{\max}) = -a$, then by Equation (\ref{eqtwofour}), $\int_{M_{\min}} c_{1}(N_{\min}) = 0$, but that contradicts the existence of a a map  $f: S^2 \rightarrow M_{\min}$, such that $\int_{S^2} (f^{*}(\omega) ) = 3$. In the remaining case that $c_{1}(N_{\max}) = -2a$, by Theorem \ref{projectiveplanetheorem} $M_{\max}$ is sympplectomorphic to $\mathbb{CP}^2$ and Equation  (\ref{eqtwofour}) gives the desired equation  $\int_{M_{\min}} c_{1}(N_{\min}) = 3$, but that contradicts the existence of a a map  $f: S^2 \rightarrow M_{\min}$, such that $\int_{S^2} (f^{*}(\omega) ) = 3$. We have shown that $b_{4}(M)>1$ is impossible.

Therefore, $b_{4}(M)=1$, that is there are precisely two fixed components, the $2$-manifold $M_{\min}$ and the $4$-manifold $M_{\max}$. By Proposition \ref{spheremaps}  there is a map $f: S^2 \rightarrow M_{\max}$, such that $\int_{S^2} (f^{*}(\omega) ) = 5.$ Let $a$ be the generator of $H^{2}(M_{\max},\mathbb{Z})$ modulo torsion, such that $ca = [\omega]$ for some $c>0$. Then set, $c_{1}(N_{\max}) = k'a$, $k' \in \mathbb{Z}$, note that since $b_{2}(M_{\max})=1$ and $M_{\max}$ is a symplectic $4$-manifold $\sigma(M_{\max})=1$ and therefore by the Hirzebruch signature theorem $\int_{M_{\max}} c_{1}(M_{\max})^2 = 9$. Therefore $c_{1}(M_{\max}) = -3a$ or $c_{1}(M_{\max}) = 3a$. It will be shown that the first case is impossible, suppose for a contradiction that  $c_{1}(M_{\max}) = -3a$.  Write $PD(f_{*}([S^2])) = k$, then it holds that $(-3+ k')k = 5$. Since $(-3+k')a = [\omega]$, it holds that $-3+k'>0$, therefore $-3+k'=1,5$ i.e. $k' = 4,8$. In such a case $\int_{M_{\max}} c_{1}(N_{\max})^2=16,64$ respectively, the  by Equation (\ref{eqtwofour}), $\int_{M_{\min}} c_{1}(N_{\min}) = 17,65$ respectively and therefore $[\omega]|_{M_{\min}}$ is $19$ or $67$ times the generator of the second cohomology. However there is a map $f: S^2 \rightarrow M_{\min}$, such that $\int_{S^2} (f^{*}(\omega) ) = 5,$ which gives a contradiction.

 In the case where $c_{1}(M_{\max}) = 3a$, the equation becomes $(3+k')k=5$ where $3+k'>0$ therefore $(3+k') =1,5$, and correspondingly $c_{1}(N_{\max}) = -2a$ or  $c_{1}(N_{\max}) = 2a$ and $\int_{M_{\max}} c_{1}(N_{\max})^2 =4$. Finally, note that since $c_{1}(M_{\max})$ is a positive multiple of $a$ and therefore of $[\omega]$ by the definition of $a$, $M_{\max}$ is symplectomorphic to a del Pezzo surface with $b_{2}=1$ and so to $\mathbb{CP}^2$.  
\end{proof}

\subsection{ $(d_1,d_2) = (4,4)$}

In this subsection the case $(d_1,d_2)=(4,4)$ is analysed, the primary example in this case is an action on $Q^4$ described below.

The image of the Pl\"{u}cker embedding $Gr(2,4) \rightarrow \mathbb{P}(\wedge^2(\mathbb{C}^4)) =  \mathbb{P}^5$ is a smooth quadric, denoted $Q^4$. In the following example we study a torus action on $Q^4$ constructed via the above identification.

\begin{example} \label{quadricexample}
Let $X =Q^4 \cong   Gr(2,4)$, and consider the $\mathbb{C}^*$-action induced by $$z.(z_1,z_2,z_3,z_4) = (z z_1,z_2,z_3,z_4).$$ Then $X^{\mathbb{C}^*} = \Pi_1 \cup \Pi_2$, where $\Pi_i$ are planes for $i=1,2$.
\end{example}

Let $H = \{(0,z_2,z_3,z_3) \mid (z_2,z_3,z_3) \in \mathbb{C}^3\} \subset \mathbb{C}^4$. Then $\Pi_1=\{\langle v_1,v_2 \rangle \in  Gr(2,4) \mid v_1,v_2 \in H \} \cong Gr(2,3) \cong \mathbb{P}^2$ is fixed by the linear action therefore is fixed. Moreover, $$\Pi_{2} = \{\langle (1,0,0,0), (0,z_2,z_3,z_4 ) \rangle \in Gr(2,4) \mid (z_2,z_3,z_4 ) \in \mathbb{C}^3  \} \cong \mathbb{P}^2$$ is also fixed.

Now, to prove that the action is semi-free, by Lemma \ref{algebraicrestriction}  it suffices to show that the action of $S^1 \subset \mathbb{C}^*$ is semi-free.  Note that $X$ is naturally isomorphic to the space of lines in $\mathbb{P}^3$ denoted $\mathbb{G}(1,3)$, and via this identification the $S^1$-action is induced by the action on $\mathbb{P}^3$ given by the above linear action i.e. $$z.(z_1,z_2,z_3,z_4) = (z z_1,z_2,z_3,z_4).$$ By direct computation in affine charts, this $S^1$-action on $\mathbb{P}^3$ is semi-free and the fixed point set of the action is $q:=[1:0:0:0]$ and the hyperplane $H$. $\mathbb{G}(1,3)^{\mathbb{C}^*}$, is equal to the union of lines contained in $H$ and lines containing $q$.

 Let $L$ be a line which is not fixed by the action, i.e. $L$ is not contained in $H$  and $L$ doesn't contain $q$ so it intersects  $H$ in a point $p_{L} \in L \cap H$. Let $p \in L \setminus \{p_{L}\} $, in particular $p$ is not contained in the fixed point set of the action on $\mathbb{P}^3$. Let $L'$ be the line through $p$ and $q$, note that $S^1.p$ is contained in $L'$. Let $g$ be a non-trivial element of $S^1,$ then $g.p$ is in $L'$ and $L \cap L' = p$, so $g.p$ cannot be contained in $L$. Therefore, $g.L \neq L$. That is, every orbit of the $S^1$-action on $\mathbb{G}(1,3) \cong Gr(2,4)$ not intersecting the fixed point set is a free orbit, i.e. the $S^1$-action is semi-free, therefore the $\mathbb{C}^*$-action is semi-free.

By Lemma  \ref{algebraicrestriction} the restriction of the $\mathbb{C}^*$-action to $S^1$ is Hamiltonian for a K\"{a}hler form $\omega$ such that $c_1 = [\omega]$. Letting $M=Q^4$, by Theorem \ref{wsf} the Hamiltonian $H$ may be normalized so that $H(p) =- \sum w_i$ for all $p \in M^{S^1}$. Since $M^{S^1}$ has only two fixed components $M_{\min},M_{\max} $ are both planes in $Q^4$, and by the semi-free condition the weights are $\{1,1,0,0\}$ and $\{-1,-1,0,0\}$ and $H(M_{\min}) = -2, H(M_{\max}) = 2$.

Let $\Pi$ be a linear plane contain in a smooth quadric $4$-fold $Q \subset \mathbb{P}^5$. There is a short exact sequence of normal bundles: $$0 \rightarrow N_{\Pi,Q} \rightarrow N_{\Pi,\mathbb{P}^5} \rightarrow N_{Q,\mathbb{P}^5} \rightarrow 0.$$Therefore the above exact sequence gives the equation of total Chern classes, it follows that $c(N_{\Pi,Q})  c(N_{Q,\mathbb{P}^5}) = c(N_{\Pi,\mathbb{P}^5}) .  $ 
We have the standard identifications $N_{\Pi,\mathbb{P}^5} \cong \mathcal{O}(1) \oplus \mathcal{O}(1)  \oplus \mathcal{O}(1)$ and  $N_{Q,\mathbb{P}^5} \cong \mathcal{O}(2)$, substituting these expressions into the above equation gives:  $c(N_{\Pi,Q}) = 1+h + h^2, $ where $h \in H^{2}(\Pi,\mathbb{Z})$ is the hyperplane class. Note that this is the total Chern class of $T\mathbb{P}^2 \otimes \mathcal{O}(-1)$ therefore $N_{\Pi,Q}$ and $T\mathbb{P}^2 \otimes \mathcal{O}(-1)$ are equivalent as complex vector bundles.

Another example fitting into this case is a Fano $4$-fold of index $2$ and genus $8$, this forms a family of Mukai fourfolds denoted $X^{m}_{8}$. I would like to thank Alexander Kuznetsov for providing this example. Since $b_{4}(X^{m}_8)=8$ Kuznetsov's example answers \cite[Question 15.3]{PZ3} negatively.

\begin{example} \label{kuznetsov}
There exists a smooth prime Fano $4$-fold $X$ in the family $X^{m}_{8}$ having a semi-free $\mathbb{C}^*$-action. Moreover $X^{\mathbb{C}^*} = \Pi_1 \cup \Pi_2 \cup \{p_1,\ldots,p_6\}$, where $\Pi_i$ are planes. In particular,  $X^{m}_{8}$ contains compactifications of $\mathbb{C}^4$.
\end{example}
\begin{proof}
Let $V_0$ and $V_1$ be $3$-dimensional vector spaces. Let $A$ be a $4$-dimensional subspace: $$A \subset V_0^* \otimes V_1^* \subset \wedge^2 (V_0 \oplus V_1)^*.$$  Then define $X_A := Gr(2,V_0 \oplus V_1) \cap \mathbb{P}(A^{\perp})$, where $$A^{\perp} = \{v \in \wedge^2 (V_0 \oplus V_1) \mid a(v) =0, \; \forall a \in A\}.$$ Set $V = V_0 \oplus V_1$. Consider the Segre subvariety $\mathbb{P}(V_0) \times \mathbb{P}(V_1)  \subset \mathbb{P}(V_0 \otimes V_1)$, and its secant variety $\Sigma  \subset \mathbb{P}(V_0 \otimes V_1)$, which is the determinantal cubic hypersurface. Note that $Sing(\Sigma) = \mathbb{P}(V_0) \times \mathbb{P}(V_1) $.

\textbf{Claim.} $X_{A}$ is smooth if and only if the cubic surface $\mathbb{P}(A) \cap \Sigma$ is smooth. In particular, $X_{A}$ is smooth if $A$ is general.

First the action of $\mathbb{C}^*$ on $X_A$ is defined, then the claim is proved. Consider the $\mathbb{C}^*$-action on $V_0$ and $V_1$, equal to the weight space of weight $0$ and $1$ respectively. Note that the induced $\mathbb{C}^*$-action on $\wedge^{2}(V)$ has weight $0,1,2$ on the subspaces $\wedge^2(V_0)$, $V_0 \otimes V_1$ and $\wedge^2(V_1)$, respectively. Therefore the subset $A \subset \wedge^2 (V_0^* \oplus V_1^*)$ is $\mathbb{C}^*$-invariant, hence $X_A$ is invariant by the induced $\mathbb{C}^*$-action on $Gr(2,V)$.

Note also that the fixed locus of the $\mathbb{C}^*$-action on $\mathbb{P}(\wedge^2(V))$ is:
$$ \mathbb{P}(\wedge^2(V_0)) \sqcup \mathbb{P}(V_0 \otimes V_1) \sqcup \mathbb{P}( \wedge^2(V_1)). $$
The first and last two components are planes, they are contained in $Gr(2,V)$ and $\mathbb{P}(A^\perp)$, hence in $X_A$.  On the other hand, the intersection of $\mathbb{P}(V_0 \otimes V_1)$
with $Gr(2,V)$ is the Segre variety $\mathbb{P}(V_0) \times \mathbb{P}(V_1)$,
and its intersection with $\mathbb{P}(A^\perp)$ is a general linear
section of $\mathbb{P} (V_0) \times \mathbb{P}(V_1)$ of codimension $4$, which
consists of $\deg(\mathbb{P}(V_0) \times \mathbb{P}(V_1))  = 6$ points.

Finally, note that the $\mathbb{C}^*$-action is semi-free. Indeed,
the subset of points in $\mathbb{P}(\wedge^2V)$ that have finite non-trivial
stabilizer in $\mathbb{C}^*$ is: $$\mathbb{P}(\wedge^2V_0 \oplus \wedge^2V_1) \setminus (\mathbb{P}(\wedge^2V_0) \cup \mathbb{P}(\wedge^2V_1)).$$ This locally closed subset of $\mathbb{P}(\wedge^2V)$
has empty intersection with $Gr(2,V)$, hence also with $X_A$.

\textbf{Proof of claim.} Only the ``if" direction is proven,
because the other direction is irrelevant now. First, let us check that $X_A$ is singular if and only if
the cubic surface

$$Y_A := Pf(V^*) \cap \mathbb{P}(A)$$
is singular, where $Pf(V^*) \subset \mathbb{P}(\wedge^2V^*)$ is
the Pfaffian cubic hypersurface. Indeed, if $X_A$ is not
smooth, there is a point $[U] \in X_A$ and a point $[a] \in \mathbb{P}(A)$
such that the hyperplane section $H_a \cap Gr(2,6)$ is singular
at $[U]$. Then the skew-symmetric form corresponding to $a$ is
degenerate and $U$ is contained in its kernel. In particular,
$[a] \in Pf(V^*)$.

Now, if the rank of $a$ as a skew-form is $2$, then $Pf(V^*)$
is singular at $[a]$, hence also $Y_A$ is singular at $[a]$.
So, assume that the rank of $a$ is $4$. Then $U$ is equal
to the kernel of $a$.

Since $[U] \in X_A$, it follows that all skew-forms $a' \in A$
must vanish on $[U]$. But the tangent space to $Pf(V^*)$ at $[a]$
is exactly the space of all skew-forms that vanish at $U$.
This shows that the tangent space to $Y_A$ at $[a]$ has
dimension 3, hence $Y_A$ is singular at $[a]$.

Next, let us show that $Y_A = \mathbb{P}(A) \cap \Sigma$. This is easy,
because the determinant of skew-forms in $\mathbb{P}(V_0^* \otimes V_1^*)$
is equal to the square of their determinants, hence their
Pfaffian is equal to the determinant up to sign.

Finally, let us show that $\mathbb{P}(A) \cap \Sigma$ is smooth
for general A. Indeed, there is a version of Bertini's
Theorem saying that a general hyperplane section $Y \cap H$
of a singular variety Y is smooth away from $Sing(Y) \cap H$.
Applying this five times to $Y = \Sigma$, we conclude that
for $A$ general $Y_A$ is smooth away from:
$$\mathbb{P}(A) \cap \mathbb{P}(V_0^*) \times \mathbb{P}(V_1^*).$$
But $\dim(\mathbb{P}(V_0^*) \times \mathbb{P}(V_1^*)) = 4$, so its general
linear section of codimension $5$ is empty \footnote{
The proof of this claim also follows from a general fact about homological projective duality \cite[Theorem 6.3]{Ku2}. Namely, that if $X$ is homologically projective dual to $Y$ then $X_{A}$ is smooth $\iff$ $Y_{A}$ is smooth, it does not appear explicitly in the literature although it is mentioned in \cite[Page 642]{Ku2}. To prove this; if $X_{A}$ is smooth then $D^{b}(X_A)$ is smooth and proper. Hence its semi-orthogonal component $\mathcal{C}_{A}$ is smooth and proper. Hence $D^{b}(Y_A)$ is smooth and proper because it has a semi-orthogonal decomposition where all components are smooth and proper, hence $Y_{A}$ is smooth. The proof of the converse is similar. }.

To prove the fact about compactifications. Consider $Gr(2,V_1)$, denote by $\mathcal{U}_1$
its tautological bundle, and consider the morphisms of vector bundles

$$V_0 \otimes \mathcal{U}_1 \to V_0 \otimes V_1 \otimes \mathcal{O} \to \wedge^2V \otimes \mathcal{O}
\to A \otimes \mathcal{O},$$ where the first arrow is induced by the tautological embedding of $\mathcal{U}_1$,
the second is induced by the natural inclusion, and the third is the dual of $A \to \wedge^2V^*$.

The source bundle has rank $6$ and the target rank $4$. Moreover, the dual
of the first bundle is globally generated. Therefore, if $A$ is general,
the morphism is surjective, and its kernel is locally free of rank $2$.
By definition the total space of the kernel is contained in $X_A$.
But by the Quillen–Suslin theorem the kernel is trivial over $\mathbb{A}^2$,
which gives an embedding of $\mathbb{A}^4$ into $X_A$.
\end{proof}

In the following proposition, it is shown that for a closed symplectic $8$-manifold having $b_{2}(M)=1$ and having a semi-free Hamiltonian $S^1$-action with extremal fixed components  both having dimension $4$, then the extremal fixed components are symplectomorphic to $\mathbb{CP}^2$, moreover $c_1^2$ and $c_2$ of the normal bundle of each extremal component is equal to $1$, as in Example \ref{quadricexample}.

\begin{proposition} \label{fourfourcase}
Let $(M,\omega)$ be a closed simply connected symplectic $8$-manifold having $b_{2}(M)=1$ and having a semi-free Hamiltonian $S^1$-action. Suppose that $(d_1,d_2) = (4,4)$. Let $N_{\min},N_{\max}$ denote the normal bundle of $M_{\min},M_{\max}$ respectively. Then, $M_{\min}$ and $M_{\max}$ are symplectomorphic to $\mathbb{CP}^2$,  $$\int_{M_{\min}} c_{2}(N_{\min}) + \int_{M_{\max}}c_{2}(N_{\max}) =b_{4}(M),$$ $$\int_{M_{\min}} c_{1}(N_{\min})^2 = \int_{M_{\max}} c_{1}(N_{\max})^2 =1.$$ Finally, if $b_{4}(M) > 2$ then $c_{1}(N_{\min}),c_{1}(M_{\max})$ are equal to the class$-h \in H^{2}(\mathbb{CP}^2,\mathbb{Z})$, where $h$ denotes the hyperplane class.
\end{proposition}
\begin{proof} 

By Lemma \ref{symfano} the symplectic form may be multiplied by a positive constant so that $c_1 = [\omega]$. Let $H$ be the Hamiltonian produced by Theorem \ref{wsf}. Since $\dim(M_{\min}) = 4$ and the action is semi-free the weights at $M_{\min}$ are $\{0,0,1,1\}$ and $H(M_{\min}) = -2$ $\dim(M_{\max}) = 4$ and the action is semi-free the weights at $M_{\max}$ are $\{-1,-1,0,0\}$  and $H(M_{\max}) = 2$.

By Poincar\'{e} duality and the assumptions $b_{2}(M)=b_{6}(M)=1$. Therefore by Proposition \ref{homologyloc} the following formulas hold: 

\begin{equation} \label{fourfourlocone}  
1= b_{2}(M) = b_{2}(M_{\min}) +  \sum_{F \subset M^{S^1}, F \neq M_{\min}} b_{2-2\lambda_{F}}(F).  \end{equation}

Since $M_{\min}$ is a symplectic $4$-manifold, $b_{2}(M_{\min})=1$. Similarly

\begin{equation} \label{fourfourloctwo}  
1= b_{6}(M) = b_{2}(M_{\max}) +  \sum_{F \subset M^{S^1}, F \neq M_{\max}} b_{6-2\lambda_{F}}(F).  \end{equation}

Since $M_{\max}$ is a symplectic $4$-manifold, $b_{2}(M_{\max})=1$. Therefore Equations (\ref{fourfourlocone}),(\ref{fourfourloctwo}) imply that for any non-extremal fixed component $F$, $b_{i-2\lambda_{F}}(F) = 0$ for $i=2,6$.  Applying the formulas for $b_{0}(M),b_{8}(M)$ in a similar way show that for any non-extremal fixed component $F$ $b_{i-2\lambda_{F}}(F) = 0$ for $i=0,8$.  The equations  $b_{i-2\lambda_{F}}(F) = 0$ for $i=0,2,6,8$ imply $\lambda_{F}=2$, and that $\dim(F)=0$.  Therefore since the action is semi-free the weights of the action along $F$ are $\{-1,-1,1,1\}$ for any non-extremal fixed component. Putting this all together, if $N_{2}$ denotes the number of fixed points with index $2$, then Proposition \ref{homologyloc} gives 
\begin{equation} \label{fourfourlocthree}  
b_{4}(M) = b_{0}(M_{\max}) + b_{4}(M_{\min}) + N_{2} = 2 +N_{2}.  \end{equation}

  By \cite[Theorem 3.4]{L2} the intersection form of $M$ is positive definite, i.e. $\sigma(M) = b_{4}(M)$, where $\sigma$ denotes the signature.  By \cite[Main Result]{JO}, $$M^{\mathbb{Z}_2}.M^{\mathbb{Z}_2} = \sigma(M) = b_{4}(M). $$
Then,  since the action is semi-free $M^{\mathbb{Z}_2} =M^{S^1}$, moreover the since only components of dimension at least $4$ can have non-trivial self-intersection, and for a four-dimensional symplectic submanifold the self-intersection is the second Chern class of the normal bundle: $$M^{\mathbb{Z}_2}.M^{\mathbb{Z}_2} = \int_{M_{\min}} c_{2}(N_{\min}) + \int_{M_{\max}} c_{2}(N_{\max}) = b_{4}(M).$$ 

By Lemma \ref{abbvone} and Equation (\ref{fourfourlocthree}):$$\int_{M_{\min}} c_{1}(N_{\min})^2 -\int_{M_{\min}} c_{2}(N_{\min}) + \int_{M_{\max}}c_{1}(N_{\max})^2 - \int_{M_{\max}} c_{2}(N_{\max})  = 2-b_{4}(M). $$ Substituting the above equation gives $$\int_{M_{\min}} c_{1}(N_{\min})^2 +  \int_{M_{\max}}c_{1}(N_{\max})^2   = 2.$$ Since $b_{2}(N_{\min})=1$ let $a$ be a generator of $H^{2}(M_{\min},\mathbb{Z})$ because by Poincar\'{e} duality the intersection pairing is a non-degenerate symmetric bilinear form, $\int_{M_{\min}} a^2 =\pm 1$. writing $a = k [\omega]|_{N}$ for some $k \in \mathbb{R}$, then $\int_{M_{\min}} a^2 = \int_{M_{\min} }k^2 [\omega]^2 >0$, therefore $\int_{M_{\min}} a^2=1$, therefore  $\int_{M_{\min}}   c_{1}(N_{\min})^2$ is a square, similarly $\int_{M_{\max}}  c_{1}(N_{\max})^2$ is a square. The only pair of squares that sum to $2$ is $\{1,1\}$. Therefore $$\int_{M_{\min}} c_{1}(N_{\min})^2 =  \int_{M_{\max}} c_{1}(N_{\max})^2   = 1.$$

Let $F$ be either of the extremal fixed components having normal bundle $N$, which have dimension $4$ by assumption, then since by Proposition \ref{homologyloc} they have $b_{2}(F)=1$ and they are symplectic $4$-manifolds they satisfy $\sigma(F)=1$. Moreover since $M$ and hence $F$ \cite{Li} is simply connected $b_{1}(F)=0$  and $\chi_{top}(F)=3$. Therefore by the Hirzebruch signature theorem $\int_{F} c_{1}(F)^2 = 9.$ Let $a$ be the integral generator of the second homology which is a positive multiple of $[\omega]|_{F}$. Then $c_{1}(F) = \pm 3a$. In fact, $c_{1}(F)=-3a$ is impossible, because by the equation $\int_{F} c_{1}(N_{F})^2 =1$ shown above, then $[\omega]|_{F} = (-3 + k')a$ where $k'=\pm1$ which contradicts the definition of $a$. Therefore $c_{1}(F)=3a$, in particular $F$ is positive monotone, and therefore since $b_{2}(F)=1$ by Theorem \ref{projectiveplanetheorem} $F$ is symplectomorphic to $\mathbb{CP}^2$.

If $b_{4}(M)>2$ then by Equation (\ref{fourfourlocthree}) there exists an isolated fixed point $p$ with $\lambda_{p}=2$. By Proposition \ref{spheremaps} there are maps $f_{1} : S^2 \rightarrow M_{\min}$,   $f_{2}: S^2 \rightarrow M_{\max}$ such that $\int_{S^2} f_{1}^{*}\omega = 2$. $\int_{S^2} f_{2}^{*}\omega = 2$. This shows that the restriction of $c_{1}(M) = [\omega]$ to $M_{\min},M_{\max}$ cannot be divisible by $4$, so the only remaining possibility iss the claimed property i.e. that $c_{1}(N_{\min}),c_{1}(M_{\max})$ are equal to the class$-H \in H^{2}(\mathbb{CP}^2)$.\end{proof}

Let $(M,\omega)$ be a closed symplectic $8$-manifold with a semi-free Hamiltonian $S^1$-action. Combining the results of this section, gives quite a lot of topological consequences for the topology of $M$, which are worked out fully in Section \ref{proofsymplectic} after proving some more preparatory results. The following theorem, states the topological restrictions which are implied immediately from the results in this section.

\begin{theorem} \label{intermediate}

Let $(M,\omega)$ be a closed simply connected symplectic $8$-manifold with $b_{2}(M)=1$ and a semi-free Hamiltonian $S^1$-action. Then the following hold:
\begin{itemize}
\item[a).] $M$  has Todd genus $1$.
\item[b).] All non-isolated fixed  components are symplectomorphic to either: $$ \mathbb{CP}^1,\;\mathbb{CP}^2, \; \mathbb{CP}^1 \times \mathbb{CP}^1,\; \mathbb{CP}^3,$$ with the (product of) Fubini-Study symplectic forms.
\end{itemize}
\end{theorem}
\begin{proof}

If $(d_1,d_2) = (0,0),$ By Proposition \ref{oneonecase}, in addition to the isolated fixed points $M_{\min},M_{\max}$, there is one non-extremal fixed component symplectomorphic to $N=\mathbb{CP}^1 \times \mathbb{CP}^1$, statement b). follows. Since $M_{\min}$ is isolated, the fact that $M$ has Todd genus $1$ follows from the localization formula for the Hirzebruch polynomial. 

If $(d_1,d_2) = (0,4)$ Then, since $M_{\min}$ is isolated point the Todd genus of $M$ is $1$ by the localization formula for the Hirzebruch polynomial, proving a). Statement b).  follows immediately from Lemma \ref{zerofourcaseone} and Proposition \ref{zerofourcasetwo}. 

If $(d_1,d_2) = (0,6)$ Then, since $M_{\min}$ is isolated pointthe Todd genus of $M$ is $1$ the localization formula for the Hirzebruch polynomial, proving a). By Lemma \ref{zerosixcase} $M_{\max}$ is symplectomorphic to $\mathbb{CP}^3$ and there are no non-extremal fixed components, proving b).

If $(d_1,d_2) =(2,4),$ Then, by Proposition \ref{twofourcase} $M_{\min}$ is symplectomorphic to $\mathbb{CP}^1$ a therefore the Todd genus of $M$ is $1$ by the localization formula for the Hirzebruch polynomial, proving a).  By Proposition  \ref{twofourcase}  $M_{\max}$ is symplectomorphic to $\mathbb{CP}^2$  and there are no non-extremal fixed components proving b).

If  $(d_1,d_2) =(4,4)$ Then, by Proposition \ref{fourfourcase} $M_{\min},M_{\max}$ are symplectomorphic to $\mathbb{CP}^2$  therefore  the Todd genus of $M$ is $1$ the localization formula for the Hirzebruch polynomial, proving a). Moreover by Proposition \ref{fourfourcase}, every non-extremal fixed point isolated proving b). 
\end{proof}

\section{Proof of Theorem \ref{symplecticmain}} \label{proofsymplectic}
\subsection{Extremal submanifolds encode the index}
The following topological lemma about sphere bundles is a consequence of the Gysin sequence. The construction of the disk and sphere bundles from a vector bundle are constructed via a bundle metric but do not depend on the metric, so the metric is suppressed in the following.

\begin{lemma} \label{gysinlemma}
Let $N$ be a simply connected, orientable, closed manifold with torsion-free integral homology and $E$ a complex vector bundle with rank $k>1$ over $N$. Then, let $D(E)$ be the unit disk bundle and $S(E)$ the unit sphere bundle respectively. Then the inclusion $i$ induces an isomorphism $$i_{*} : H_{2}(S(E),\mathbb{Z}) \rightarrow  H_{2}(D(E),\mathbb{Z}). $$
\end{lemma}

\begin{proof}
Denote by $i: S(E) \rightarrow D(E)$ the inclusion. Let $\pi_{S},\pi_{D}$ denote the projection to the base of $S(E)$ and $D(E)$, clearly $\pi_{D}$ is a homotopy equivalence and $\pi_{S} =  \pi_{D}  \circ i$. Therefore, the claim is equivalent to showing that $\pi_{S} $ induces an isomorphism $$(\pi_{S})_{*} : H_{2}(S(E),\mathbb{Z}) \rightarrow  H_{2}(N,\mathbb{Z}). $$

Since $E$ is a complex vector bundle by assumption, it is in particular oriented and as such it induces an orientation on fibers of $S(E)$ in a natural way. Therefore, $S(E)$ is a fiberwise oriented $S^{2k-1}$-bundle, and the Gysin sequence may be applied:
$$\ldots  \rightarrow H^{1- (2k-1)}(N,\mathbb{Z}) \rightarrow H^2(N,\mathbb{Z}) \rightarrow H^{2}(S(E),\mathbb{Z}) \rightarrow  H^{2-(2k-1)}(N,\mathbb{Z})\rightarrow \ldots  $$
By the assumption $2k>2$, $1-(2k-1)= 2-2k<0$ and $2-(2k-1)= 3-2k<0$ so the first and last terms vanish and $\pi^* :  H^2(N,\mathbb{Z}) \rightarrow H^{2}(S(E),\mathbb{Z})$ is an isomorphism. Since by assumption $N$ is simply connected, by the universal coefficients theorem, $H^{2}(N,\mathbb{Z}) \cong Hom(H_{2}(N,\mathbb{Z}),\mathbb{Z})$, and similarly for $S(E)$, proving the claim.  \end{proof}

Lemma \ref{gysinlemma} will be applied in the following proposition to closed symplectic manifolds admitting a Hamiltonian  $S^1$-action with $b_{2}(M)=1$, showing that the second homology is integrally generated by classes in the extremal submanifolds, provided they are not isolated. It is loosely inspired by \cite[Proposition 0.2]{PeSc}. This will be important later in the determining of the index in the main classification problem. The index $\iota(M)$ is defined to be the maximum integer $k$ such that $\frac{c_1(M)}{k}$ is an integral cohomology class.
\begin{proposition} \label{indexprop} Let $(M,\omega)$ be a closed symplectic manifold with a Hamiltonian $S^1$-action and $b_{2}(M)=1$ and such that each component of $M^{S^1}$ is simply connected and has torsion-free integral homology. Let $N$ be an extremal fixed component of positive dimension. Then the inclusion $i_{*}: H_{2}(N,\mathbb{Z}) \rightarrow H_{2}(M,\mathbb{Z})$ is an isomorphism. In particular, $\iota(M) = |\int_{a} c_{1}(M)|$, where $a \in H_{2}(N,\mathbb{Z})$ is a generator.
\end{proposition}
\begin{proof}
Homology groups are taken with integral coefficients throughout the proof. Without loss of generality $N=M_{\min}$ i.e. $\lambda(N) = 0$ and $\dim(N)>0$. Firstly, it is shown that for every fixed component $F$, $\lambda_{F} \neq 1$, so that $M_{\min}$ is the unique fixed component with $\lambda_{F} \leq 1$. Suppose for a contradiction there is another fixed component $F$ such that $\lambda_{F} = 1$. Then by Proposition \ref{homologyloc}: $$ b_{2}(M) =  \sum_{F \subset M^{S^1}} b_{2-2\lambda_{F}}(F) = b_{2}(M_{\min}) + b_{0}(F) + c, $$  where $c$ is non-negative integer, a contradiction.

We prove the claim inductively for the sets $H^{-1}(-\infty, c)$ as $c$ passes through critical levels, For $c = H_{\min} + \varepsilon$, $\varepsilon$ sufficiently small $H^{-1}(-\infty, c)$ is homotopy equivalent to $N$ and the proposition holds tautologically. When $c$ is a critical level for $\varepsilon$ sufficiently small then $H^{-1}(-\infty, c + \varepsilon)$ is homotopy equivalent to the space obtained as the union the negative disk bundle $N_{-}(F)$ of the fixed components $F$ such that $H(F)=c$ and $H^{-1}(-\infty, c - \varepsilon)$ \cite[Lemma 0.4]{L}. Moreover as was noted in the beginning of the proof $\lambda_{F}>1$. Since $\lambda_F$ is the rank of $N_{-}(F)$ we denote it by $\lambda_{F}=k>1$. Let $A = H^{-1}(-\infty, c + \varepsilon) $, $B = N_{-}(F)$, and $X :  =A \cup B$, so that $X$ is homotopy equivalent to  $H^{-1}(-\infty, c + \varepsilon)$. Since $F$ is simply connected and $A \cap B$ is homotopy equivalent to a $S^{2k-1}$-bundle over $F$ by the homotopy long exact sequence, it holds that $H_1(A \cap B)=0$. Then the relevant portion of the Mayer-Vietoris sequence is: $$\ldots \rightarrow H_{2}(A \cap B) \xrightarrow[]{f_1} H_{2}(A ) \oplus H_{2}(B) \xrightarrow[]{f_2} H_{2}(X) \rightarrow 0.$$

By Lemma \ref{gysinlemma} $i_{*} : H_{2}(A \cap B) \rightarrow H_{2}(A)$ is an isomorphism.  Let $f_{1} =( i_* , j_*)$,  and $f_{2} = k_* -l_*$ be the maps appearing in the above Mayer-Vietoris sequence, then by exactness of the sequence, the map induced by $f_2$  $$\tilde{f}_{2} :  (H_{2}(A ) \oplus H_{2}(B)) /f_{1}(H_{2}(A \cap B)) \rightarrow H_{2}(X) $$ is an isomorphism. Moreover the map $\tau: H_{2}(B) \rightarrow  (H_{2}(A ) \oplus H_{2}(B)) /f_{1}(H_{2}(A \cap B))$, defined by $\tau(a) = [(0,a)]$ is an isomorphism, since any element of the form $(0,a)$ in the image of $f_1 = ( i_* , j_*)$ is equal to $(0,0)$, since $i_*$ is an isomorphism. Therefore $-l_* = \tilde{f}_2 \circ \tau$ is an isomorphism, so $l_*$ is an isomorphism.
\end{proof}

The following final lemma of this section exhibits another way in which; loosely speaking, the index of a closed symplectic manifold with a Hamiltonian $S^1$-action is encoded in its fixed point set. 
\begin{lemma} \label{sphereindex}
Let $(M,\omega)$ be a closed symplectic $8$-manifold with $b_{2}(M)=1$, with symplectic form normalized so that $c_{1}(M)=[\omega]$ and a semi-free Hamiltonian $S^1$-action. Then the following hold:\begin{itemize} \item Suppose that $M_{\min},M_{\max}$ are isolated and there is a unique non-extremal fixed component $F$ of dimension $\dim(F) =4$. Then $4 \leq \iota(M)$.

\item   Suppose that $M_{\min}$ is isolated and $H^{-1}(-4,0) \cap M^{S^1} =F$ is a component of dimension $2$ with weights $\{-1,1,1,0\}$. Then $\iota(M)$ is odd. Suppose that $M_{\min}$ is isolated and $H^{-1}(-4,0) \cap M^{S^1} =p$ is a component of dimension $0$ with weights $\{-1,1,1,1\}$. Then $\iota(M) \leq 2$.  \end{itemize}
\end{lemma}
\begin{proof}
a). Firstly, let $H$ be the Hamiltonian from Theorem \ref{wsf}. Firstly, it will be shown that $\lambda_{F} = 1$ and $H(F)=0$. Suppose $\lambda_{F} \neq 1$ then there is some $i \in \{0,4\}$ whuch that $i-2\lambda_{N} =0,2, \ldots, 4$, therefore since even Betti numbers of $F$ are positive $b_{0}(M)$ or $b_{8}(M)$ is greater than $1$ by Proposition \ref{homologyloc} which is a contradiction. Hence $\lambda_{F}=1$ and therefore by the semi-free condition the non-zero weights along $F$ are $-1,1$ and so $H(F)=0$. 

Note that by assumption $H(M_{\min})=-n$ and $H(M_{\max})=n$. Pick and $S^1$-invariant almost complex structure, recall the construction of the gradient vector field and gradient flow maps from the proof of Proposition \ref{spheremaps}. The proof proceeds by constructing a cycle $S \subset M$ for which $S \cdot F = \pm 1$. Let $U \subset M$ the union of gradient flow lines with minimum in $F$. Let $0<c< n$ then the $U \cap H^{-1}(c)$ is diffeomorphic to an $S^1$-bundle over $F$,  Then since $H^{-1}[c,\infty)$ is diffeomorphic to $D^{2n},$  $U \cap H^{-1}(c)$ is a boundary of a chain $C \subset H^{-1}[c,\infty)$, therefore construct a cycle  $C_0$ by gluing $H^{-1}(-\infty, c) \cap U$ and $C$. Then  $C_0$ intersects $S$ in one point transversally, therefore $C_0 \cdot S = \pm1$. By the Duistermaat-Heckman theorem $\int_{S} \omega =n$, therefore $[S]$ generates $H_{2}(M,\mathbb{Z})$ and $\iota(M)$ is a positive multiple of $\int_{S} \omega = n$ proving the lemma.

The proof of b). uses the same idea, although here there is no need to prove that $[S]$ is primitive. There is a sphere $S \subset M$ with maximum at $F$ and $\int_{S} \omega =3$. It follows immediately that $\iota(M)$ cannot be even since otherwise  $\int_{S} \omega =3$ would be even also. Similarly if $H^{-1}(-4,0) \cap M^{S^1} =p$ is a component of dimension $0$ then there is a sphere with maximum $p$ such that $\int_{S} \omega =2$.
 \end{proof}
\subsection{Proof of Theorem \ref{symplecticmain}}
At this point we are close to being able to prove Theorem \ref{symplecticmain}, one of the main parts of which is the statement, if $(M,\omega)$ is a closed symplectic $8$-manifold with $b_{2}(M)=1$ and a semi-free Hamiltonian $S^1$-action, then $\iota(M)>1$. Next, an auxilary lemma is proven, it summarizes all of the consequenes for the index $\iota(M)$ following from the analysis of the fixed point set in Section \ref{excases} and the general results about the index proved in this section.  Let $h \in H^2(\mathbb{CP}^2,\mathbb{Z})$ represent the hyperplane class.

\begin{lemma} \label{squares}
Let Let $(M,\omega)$ be a closed  simply connected symplectic $8$-manifold with $b_{2}(M)=1$ and having a semi-free Hamiltonian $S^1$-action. Then $\iota(M) > 1$, in particular:\begin{itemize}\item[a).] If $(d_1,d_2) = (0,0)$ $b_{4}(M)=2$ and $\iota(M)=4$, if $(d_1,d_2) = (2,4),(0,6)$ then $b_{4}(M) =1$ and $\iota(M) = 5$.
\item[b).] If $(d_1,d_2) = (0,4)$ then $M_{\max}$ is symplectomorphic to $\mathbb{CP}^2$. Moreover it holds that $$b_{4}(M) = \int_{M_{\max}} c_{2}(N_{\max})$$ and $2 \leq \iota(M) \leq 4$. If $\iota(M) = 2$ or $b_{4}(M) >2$ then $c_{1}(N_{\max}) = -h$. If $M^{S^1}$ does not contain a $2$-dimensional component then $\iota(M)=2$.
\item[c).] If $(d_1,d_2) =(4,4)$ then $M_{\min},M_{\max}$ are symplectomorphic to $\mathbb{CP}^2$. Moreover it holds that $$b_{4}(M) =  \int_{M_{\min}} c_{2}(N_{\min}) +  \int_{M_{\max}} c_{2}(N_{\max}),$$ and $\iota(M) \in \{2,4\}$. If $\iota(M)=2$ or $b_{4}(M) > 2$ then $c_{1}(N_{\min}) = -h$ and $c_{1}(N_{\max}) = -h$. 
\end{itemize}
\end{lemma}
\begin{proof}

Throughout we denote by $N_{\min}$ resp. $N_{\max}$ the normal bundles of $M_{\min}$ and $M_{\max}$. By Lemma \ref{ruleoutone} and Lemma \ref{ruleouttwo}, $$(d_1,d_2) = (0,0),(0,4),(0,6),(2,4),(4,4),$$ are the possible values, these are analysed seperately. 

If $(d_1,d_2) = (0,0),$ By Proposition \ref{oneonecase}, in addition to the isolated fixed points $M_{\min},M_{\max}$, there is one non-extremal fixed component symplectomorphic to $N=\mathbb{CP}^1 \times \mathbb{CP}^1$. By Proposition \ref{homologyloc} $b_{4}(M)=2$ and by Lemma \ref{sphereindex}, $4 \leq \iota(M) $.

If $(d_1,d_2) = (0,6)$ By Lemma \ref{zerosixcase} $M_{\min}$ is isolated point and $M_{\max}$ is symplectomorphic to $\mathbb{CP}^3$ and $c_{1}(N_{\max})= H$. Therefore by Proposition \ref{homologyloc} $b_{4}(M)=1$ and by Proposition \ref{indexprop} $\iota(M) = 5$. 

If $(d_1,d_2) =(2,4),$ Then, by Proposition \ref{twofourcase} $M_{\min}$ is symplectomorphic to $\mathbb{CP}^1$ and $\int_{M_{\min}} = c_{1}(N_{\min}) = 3$, $M_{\max}$ is symplectomorphic to $\mathbb{CP}^2$. By Proposition \ref{homologyloc}  $b_{4}(M)=1$.  Since  $\int_{M_{\min}} = c_{1}(N_{\min}) = 3$ $c_{1}(M)|_{M_{\min}}$ is $5$ times the positive generator therefore by Proposition \ref{indexprop}, $\iota(M) = 5$. Therefore statement a). is proved.

If $(d_1,d_2) = (0,4)$, the claimed equations about Chern classes are a part of Lemma \ref{zerofourcaseone} and Proposition \ref{zerofourcasetwo}. Moreover, it was shown $M_{\max}$ is symplectomorphic to $\mathbb{CP}^2$ and $c_{1}(N_{\max}) = aH$, where $|a| \leq 1$, therefore $c_{1}(M)|_{M_{\max}} = bH$, where $2 \leq b \leq 4$.  therefore by Proposition \ref{indexprop}, $ 2 \leq \iota(M) \leq 4$, and $\iota(M)=2$ exactly when $b=2$, which is equivalent to $c_{1}(M_{\max})=-H$. If $M^{S^1}$ does not contain a $2$-dimensional fixed component of dimension $2$ then $c_{1}(N_{\max})=\pm h$ by Lemma \ref{zerofourcaseone}, therefore by Proposition \ref{indexprop} $\iota(M)=2,4$, and furthermore by Lemma \ref{sphereindex} $\iota(M)=2$,  proving b). 

If  $(d_1,d_2) =(4,4)$, the claimed equation about Chern classes was given in Proposition \ref{fourfourcase}.  Then, by Proposition \ref{fourfourcase} $M_{\max}$ is symplectomorphic to $\mathbb{CP}^2$ and $c_{1}(N_{\max}) = aH$, where $a = \pm 1$, therefore $c_{1}(M)|_{M_{\max}} = bH$, where $2 \leq b \leq 4$. Therefore by applying Proposition \ref{indexprop} to $M_{\max}$, $\iota(M) \geq 2$,  and $\iota(M)=2$ exactly when $b=2$, which is equivalent to $c_{1}(N_{\max})=-H$, which in turn impies $c_{1}(N_{\min})=-H$ after applying  Proposition \ref{indexprop}  to the reversed circle action,  proving c). \end{proof}

The following theorem summarizes the results proven in a symplectic context at this point.

\begin{theorem} \label{symplecticmain}
Let $(M,\omega)$ be a closed  simply connected symplectic $8$-manifold with a semi-free Hamiltonian $S^1$-action. Then the following hold:
\begin{itemize}

\item[a).] $M$  has Todd genus $1$.
\item[b).] All non-isolated fixed  components are symplectomorphic to either: $$ \mathbb{CP}^1,\;\mathbb{CP}^2, \; \mathbb{CP}^1 \times \mathbb{CP}^1,\; \mathbb{CP}^3,$$ with the (product of) Fubini-Study symplectic forms.

\item[c).] $1 < \iota(M)$.
\end{itemize}
\end{theorem}
\begin{proof}[Proof of Theorem \ref{symplecticmain}]
Statements a). and  b). follow from Theorem \ref{intermediate} and statement c). follows from Lemma \ref{squares}. \end{proof}

\section{Further Duistermaat-Heckman analysis} \label{DHA}
The final step to proving Theorem \ref{bfourbound} is to analyse the Duistermaat-Heckman measure, close to extremal fixed components. By the results of the previous sections, these are symplectomorphic to $\mathbb{CP}^k$, for some $k$.  Supposing that $M_{\max} \cong \mathbb{CP}^2$ with normal bundle denoted $N_{\max}$.Then reduced space on levels close to the extremum $H_{\max} - \varepsilon$ are symplectomorphic to $\mathbb{P}(N_{\max})$, with symplectic form described below, and by definition the Duistermaat-Heckman function is equal to the symplectic volume of this projective bundle up to a suitable constant. Since by Lemma \ref{squares} the fixed point data is classified unless there is an extremal fixed  component $M_{\min/\max}$ such that $c_{1}(N_{\min/\max}) = -H$, the lemma is stated under this restriction.

\begin{lemma}  \label{inttwo} Suppose that  $(M,\omega)$  a closed symplectic $8$-manifold with $b_{2}(M)=1$ and having a semi-free Hamiltonian $S^1$-action. Let $H$ be normalized by Theorem \ref{wsf} finally suppose $M_{\min}$ is symplectomorphic to $\mathbb{CP}^2$, $H^{-1}(-\infty,0) \cap M^{S^1} = M_{\min}$, and $c(N_{\min}) = 1 - H + k_2 H^2, $ then then the following formulas hold:  $$\int_{M_{-2+x}} [\omega_{-2+x}]^3 = 12x + 6x^2 +(1 -k_2)x^3, \;\;\; \int_{H^{-1}(-\infty,0)} \omega^{4} = 176  -  16k_2 . $$ 
Similarly if $M_{\max}$ is symplectomorphic to $\mathbb{CP}^2$, $H^{-1}(0,\infty) \cap M^{S^1} = M_{\max}$, and $c(N_{\max}) = 1 - H + k_2 H^2.$ Then the following formulas hold:   $$\int_{M_{2-x}} [\omega_{2-x}]^3 = 12x + 6x^2 +(1 -k_2)x^3, \;\;\; \int_{H^{-1}(0,\infty)} \omega^{4} = 176  -  16k_2 . $$ 
\end{lemma}
\begin{proof}
By reversing the circle action, we may just prove the first statement. Note that $H(M_{\max}) = -2$. Then the  reduced space $M_{c}$ for $c \in (-2,0)$ is diffeomorphic to $\mathbb{P}(N_{\max})$. The cohomology ring may be computed by \cite[Lemma 2.13]{LP2}, it is isomorphic to $\mathbb{Z}[\eta,\xi]/(\eta^3, \xi^2 - \eta \xi + k_{2}\eta^2)$, the top degree products of the generators are $\eta^3=0, \eta^2\xi=1, \eta \xi^2 = 1$ and $\xi^3= (1-k_{2})$.  Note that in this identification $\xi = e(H^{-1}(c))$. To see this note that by definition the level set $H^{-1}(c)$ is the unit sphere bundle of the tautological bundle of $\mathbb{P}(N_{\min}).$

Then, in level $H=-2+x$ the cohomology class of the symplectic form is $2\eta + x\xi$, therefore $$\int_{M_{-2+x}} [\omega_{-2+x}]^3 = (2\eta + x\xi)^3 = 12x + 6x^2 +(1 -k_2)x^3.$$ 

Then by Equation (\ref{dhformula}): $$\int_{H^{-1}(-\infty,0)} \omega^{4} = 4 \int_{-2}^{0}  \int_{M_{c}} [\omega_c]^3 dx = 176 -16k_2.$$\end{proof}

The following similar lemma will also be needed, giving information about the Duistermaat-Heckman function close to an isolated extremal fixed point under certain conditions.

\begin{lemma} \label{intone} Let  $(M,\omega)$  a closed symplectic $8$-manifold  with $b_{2}(M)=1$ and having a semi-free Hamiltonian $S^1$-action. Let $H$ be normalized by Theorem \ref{wsf} and finally suppose $H^{-1}(-\infty,0)$ consists of an isolated fixed point with weights $\{1,1,1,1\}$ and an isolated fixed point with weights $\{-1,1,1,1\}$, then $$\int_{H^{-1}(-\infty,0)} \omega^{4} = 240. $$ 
\end{lemma}
\begin{proof}
Applying the Duistermaat-Heckman theorem to $M_{-4+x}$, $x \in (0,2)$ gives that $$\int_{M_{-4+x}} [\omega_x]^3 =x^3,$$ for $x \in (2,4)$. For $x \in (-2,0)$,  $e(H^{-1}(c)) = h-e$ where $e$ is the exceptional divisor of the corresponding blow-up of the reduced space.  Therefore$$\int_{M_{-4+x}} [\omega_x]^3 =x^3 - (x-2)^3,$$   for $x \in (2,4)$. It follows by Equation (\ref{dhformula}) that $$\int_{H^{-1}(-\infty,0)} \omega^{4} = 4 \int_{-4}^{0}  \int_{M_{c}} [\omega_c]^3 dc =  4 \int_{2}^{4} x^3 dx = 240. $$\end{proof}

\subsection{Bounding $b_{4}(M)$}

In this subsection, we remain in a symplectic context and note that the results proved so far imply a bound on $b_{4}(M)$.

\begin{theorem} \label{bfourbound}
Let $(M,\omega)$ be a closed  simply connected symplectic $8$-manifol with $b_{2}(M)=1$ and having a semi-free Hamiltonian $S^1$-action, then $b_{4}(M) \leq 14$.
\end{theorem}
\begin{proof}
Throughout we denote by $N_{\min}$ resp. $N_{\max}$ the normal bundles of $M_{\min}$ and $M_{\max}$. By Lemma \ref{ruleoutone} and Lemma \ref{ruleouttwo}, $$(d_1,d_2) = (0,0),(0,4),(0,6),(2,4),(4,4).$$

Then if $(d_1,d_2) =  (0,0),(0,6),(2,4),$ then by Lemma \ref{squares} $b_{4}(M) \leq 2$ so the theorem is proven.

Hence assume $(d_1,d_2) = (0,4),(4,4).$ If $(d_1,d_2) = (0,4)$ and $b_{4}(M) > 2$, then by Lemma \ref{squares} $M_{\max}$ is symplectomorphic to $\mathbb{CP}^2$ and $c_{1}(N_{\max}) = 1 -H + b_{4}(M)H^2$. Note also that $H^{-1}(0,2) \cap M^{S^1} = \emptyset$. Therefore Lemma \ref{inttwo}, the Duistermaat-Heckman function on $(0,2)$ is  $$\int_{M_{2-x}} [\omega_{-2+x}]^3 = 12x + 6x^2 +(1 -b_{4}(M))x^3.$$ If $b_{4}(M)>7$ then this is negative close to level zero (i.e. $x=2$), a contradiction.

If $(d_1,d_2) = (4,4)$ and $b_{4}(M) > 2$, then by Lemma \ref{squares} $M_{\min},M_{\max}$ are symplectomorphic to $\mathbb{CP}^2$ and $c_{1}(N_{\max}) = 1 -H + k_{\max}H^2$ and  $c_{1}(N_{\min}) = 1 -H + k_{\min}H^2$ where $k_{\min}+ k_{\max}=b_{4}(M)$. Note also that $H^{-1}(0,2) \cap M^{S^1} = \emptyset$ and  $H^{-1}(-2,0) \cap M^{S^1} = \emptyset$. Therefore Lemma \ref{inttwo}, the Duistermaat-Heckman function on $(0,2)$ satisfies  $$3!DH(2-x) = \int_{M_{2-x}} [\omega_{2-x}]^3 = 12x + 6x^2 +(1 -k_{\max}(M))x^3.$$ the Duistermaat-Heckman function on $(-2,0)$ satisfies  $$3!DH(-2+x) = \int_{M_{-2+x}} [\omega_{-2+x}]^3 = 12x + 6x^2 +(1 -k_{\min}(M))x^3.$$ Assume for a contradiction $b_{4}(M) >14$, since $k_{\min}+ k_{\max}=b_{4}(M)$, at least one of $k_{\min},k_{\max}$ is greater than $7$, therefore we get the same contradiction as before, at some point the Duistermaat-Heckman function is negative, providing a contradiction.
\end{proof}
\section{Torus actions on prime Fano $4$-folds} \label{Autfan}
In this section we will prove the main result about smooth complex prime Fano $4$-folds having a semi-free torus action. 

A Fano variety is a variety with ample anticanonical bundle $-K_{X}$. For a smooth Fano variety $X$ let $\iota(X)$ denote the Fano index, that is the maximum integer $k$ such that $-K_{X} = kH$,  $H \in Pic(X)$. A Fano variety is called prime when the Picard group has rank $1$. Smooth Fano $4$-folds with $\iota(X)>1$ are classified; by Mukai and Wilson in the prime case, this was extended to arbitrary Picard rank by Wisniewski. In contrast to the situation when $\iota(X)=1$, this list is relatively short, particularly when restricting to Prime Fano $4$-folds: There are $16$ families of smooth prime Fano $4$-folds with $\iota(X)>1$, $1$ family with $\iota(X)= 4,5$ respectively, $5$ families with $\iota(X)= 3$ and $9$ families with  $\iota(X)= 2$ \cite[Tables 1-4]{CGKS}.

Let $X$ be a smooth prime Fano $4$-fold with index $2$ and let $H$ be the positive generator of the Picard group. Such varieties were classified by Mukai under the condition that $|H\cap H|$ has a smooth member \cite{Mu}. The technichal assumption was later shown to hold by Wilson \cite{W}. Let the degree be $d = \int_{X}H^4$, then the classification is usually stated in terms of the genus $g(X) := 1+ \frac{d}{2}$. Then all such fourfolds $X$ fall into a family denoted $X^{m}_{g}$ for each $g$ such that $2 \leq g \leq 10$.

The family $X_{10}^m$ is usually known via its degree $V_{18} \cong X_{10}$, and contains fourfolds admitting a $\mathbb{C}^*\times \mathbb{C}^*$-action with isolated fixed points. The following lemma gives a list of the prime Fano fourfolds with index greater than one and positive definite intersection form.  I would like to thank Alexander Kuznetsov for providing an argument showing that $h^{1,3}(X^{m}_g) = 0$  for $7 \leq g \leq 9$.

\begin{lemma} \label{fanofourfoldposdef}
Let $X$ be a smooth complex prime Fano $4$-fold with $\iota(X)>1$. $X$ has positive definite intersection form precisely when $X$ belongs to one of the following families: $$\mathbb{CP}^4,Q^4, Q_1 \cap Q_2, W_{5}, X^{m}_{7}, X^{m}_{8},X^{m}_{9}, V_{18}.$$
\end{lemma}
\begin{proof}By the Hodge index theorem and Kodaira vanishing a Prime Fano $4$-fold has positive definite intersection form if and only if $h^{1,3}(X)=0$. The proof proceeds by consulting the list of Fano $4$-folds with $\iota(X)>1$  \cite[Tables 1-4]{CGKS}. 

When $\iota(X) \geq 3$, every fourfold not in the claimed list is a complete intersection or a weighted complete intersection in  \cite[Table 1]{PS}.  Therefore by \cite[Corollary 6.1]{LW} and the values of $h^{1,3}(X)$ given in\cite[Table 1]{PS}, each of them has $h^{1,3}(X)>0$.

When $\iota(X) =2$, that is $X \cong X^{m}_{g}$ for $2 \leq g \leq 10$, then if $2 \leq g \leq 5$, $X^{m}_{g}$ is a complete intersection or a weighted complete intersection contained in \cite[Table 1]{PS}, so the proof proceeds as above. $X^{m}_{6}$ is a Gushel-Mukai fourfold,   $h^{1,3}(X^{m}_{6})=1$ by \cite[Lemma 4.1]{LM}, finally $V_{18} \cong X^{m}_{10}$.

We have narrowed the search down to the desired list, it remains to verify that each fourfold in the list has positive definite intersection form. For $\mathbb{CP}^4,Q^4, Q_1 \cap Q_2$ this follows from \cite[Corollary 6.1]{LW}.  $W_5$ and $V_{18} = X^m_{10}$ contain members having a $\mathbb{C}^*$-action with isolated fixed points, therefore there Hodge numbers are pure by \cite{BB}.  For  $7 \leq g \leq 9$ the vanishing of $h^{1,3}(X^{m}_{g})$  follows from the vanishing of the Hochschild homology of $X$ for degree greater than $0$, which in turn follows from results of \cite{Ku}.
\end{proof}

\subsection{Proof of Theorem \ref{fanomain}}
In this subsection the main result is proven.
\begin{theorem} \label{fanomain}
Let $X$ be a smooth complex prime Fano $4$-fold having a semi-free $\mathbb{C}^*$-action, then $X$ is contained in one of the families $\mathbb{P}^4,Q^4,W_5$ or $X^{m}_{8}$. 
\end{theorem}
\begin{proof}

By  Lemma  \ref{algebraicrestriction}  there is a K\"{a}hler form $\omega$ invariant by $S^1 \subset \mathbb{C}^*$ such that $c_1 = [\omega]$ and the $S^1$-action is Hamiltonian. 

Then, by  \cite[Theorem 3.4]{L2} the intersection form of $X$ is positive definite, moreover by Theorem \ref{symplecticmain} $\iota(X)>1$ therefore by Lemma \ref{fanofourfoldposdef} X belongs to one of the families $\mathbb{CP}^4,Q^4, Q_1 \cap Q_2, W_{5}, X^{m}_{7}, X^{m}_{8},X^{m}_{9}, V_{18}.$ By \cite[Theorem 3.1]{B} $Q_{1} \cap Q_2$ has finite automorphism group so can be ruled out.

Therefore, the remaining cases to rule out are precisely the cases $X = X^{m}_{7},X^{m}_{9},V_{18}$.   Recall that also $b_{4}(X^{m}_{7}) = 12$ \cite[Proposition 5.2]{PZ1}  $b_{4}(X^{m}_{9}) = 4$ \cite[Theorem 2.1]{PZ2}, and it is well-known that $b_{4}(V_{18})=2$. Moreover we can state the integral $\int_{X} c_1^4$ for each case, using the well-known formula relating the degree and genus, which in this case takes the form $d(X) = 2(g(X)-1)$, and since the index is two \begin{equation}\label{degreegenus} \int_{X^{m}_{g}} c_1^4 = 16 d(X) = 32(g(X)-1). \end{equation} We give the specific values since they will be used in the remainder of the proof: $\int_{X^{m}_g} c_1^4 = 192,224,256,288,$ for $g=7,8,9,10$ respectively.

Suppose for a contradiction there is a semi-free $\mathbb{C}^*$-action on $X_{g}^{m}$ with $ g \in \{7,9,10\}$. Since $\iota(X)=2$, by Lemma \ref{squares} we may assume that $(d_1,d_2)= (0,4), (4.4)$. Since $X$ is a Fano variety it is simply connecyed and the results of the previous sections can be applied.

Suppose that $(d_1,d_2)=(0,4)$ then by Lemma \ref{squares}  and the condition $\iota(M)=2$, it holds that $c_{1}(N_{\max}) = 1 -H + b_{4}(M)H^2$. 

Then if $M^{S^1}$ has no fixed component of dimension $2$,  by Lemma \ref{zerofourcaseone}, Lemma \ref{intone} and Lemma \ref{inttwo} $\int_{M} c_{1}^4 = 240 + 176  - 16 b_{4}(M) = 416 - 16 b_{4}(M)$. If $b_{4}(M)= 4$ then $\int_{M} c_{1}^4 = 352$, but this does not hold for $X^{m}_{9}$  by Equation (\ref{degreegenus}). If  $b_{4}(M)= 12$ then $\int_{M} c_{1}^4 = 224$, but $\int_{X^{m}_{7}} c_{1}^4= 192$  by Equation (\ref{degreegenus}). If $b_{4}(M)=2$ then $\int_{M} c_{1}^4 = 240 + 176  - 32 = 384$, but  $\int_{V_{18}} c_{1}^4= 288$.

Then if $M^{S^1}$ has a fixed component of dimension $2$, then by  Lemma \ref{sphereindex}, $\iota(M)$ is odd, which is a contradiction. So the proof is complete in the case $(d_1,d_2) = (0,4)$.

Suppose that $(d_1,d_2)=(4,4)$ and $\iota(X)=2$, then by Lemma \ref{squares} $c_{1}(N_{\max}) = 1 -H + k_{\max}H^2$ and  $c_{1}(N_{\min}) = 1 -H + k_{\min}H^2$ where $k_{\min}+ k_{\max}=b_{4}(M)$. Then  Lemma \ref{inttwo} gives that, $\int_{M} c_{1}^4 = 352  - 16 b_{4}(M)$.  If $b_{4}(M)=2$ then $\int_{M} c_{1}^4 = 352  - 32 = 320$, but  $\int_{V_{18}} c_{1}^4= 288$. If $b_{4}(M)= 4$ then $\int_{M} c_{1}^4 = 288$,  ruling out $X^{m}_{9}$ by Equation (\ref{degreegenus}). 
Finally if $b_{4}(M)=12$ then  $\int_{M} c_{1}^4 = 160$ which is not the case for $X^{m}_{7}$  by Equation (\ref{degreegenus}).
\end{proof}
Finally, we prove a slightly more precise version of the previous theorem, the result follows quite quickly from the main proof.

\begin{theorem} \label{fpequivalent}
Let $X$ be a smooth complex prime Fano $4$-fold with a semi-free $\mathbb{C}^*$-action. Then, one of the following cases occurs: \begin{itemize}
\item[a).] $X \cong \mathbb{P}^4$ and the action is FP-equivalent to the one of the actions given in Example \ref{projectiveone} and Example \ref{exampletwofour}.
\item[b).] $X \cong Q^4$ and the action is FP-equivalent to one of the ones given in Example \ref{quadricone} and Example \ref{quadricexample}.
\item[c).] $X \cong W_5$ and the action is FP-equivalent to the action described in Example  \ref{Wexample}.

\item[d).] $X$ is contained in the family $ X^{m}_{8}$. $X$ is FP-equivalent to the action described in Example \ref{kuznetsov}.
\end{itemize}

\end{theorem}
\begin{proof} Since $X$ is Fano it is simply connected and the results of the previous sections may be applied. By Theorem \ref{fanomain}, $X \cong W_5,Q^4,\mathbb{P}^4$ or $X$ is in the family $ X^m_{8}$. If $X \cong W_5$ then $\iota(X)=3$, and $b_{4}(X) =2$ then by Lemma \ref{squares} $(d_1,d_2) = (0,4)$. Then if $M^{S^1}$ consists of $M_{\max}$ and isolated fixed points, by Lemma \ref{squares} $\iota(M)=2,4$, which is a contradiction. If $M^{S^1}$ contains a two-dimensional fixed component, then by Lemma \ref{squares} $c_{1}(N_{\max})=0$, therefore applying Proposition \ref{zerofourcasetwo} the action is FP equivalent to the one given in Example \ref{Wexample}.

If $X \cong Q^4$, then by Lemma \ref{squares} $(d_{1},d_2) = (0,0),(0,4),(4,4)$. If $(d_1,d_2)=(0,0)$ by Proposition \ref{oneonecase} the action is FP-equivalent to Example \ref{quadricone}. If $(d_1,d_2) = (4,4)$ by Lemma \ref{squares} the action is FP equivalent to   Example \ref{quadricexample}. If $(d_1,d_{2}) = (0,4)$, if $M^{S^1}$ consists of $M_{\max}$ and isolated fixed points, by Lemma \ref{squares} $\iota(M)=2$, which is a contradiction. If $M^{S^1}$ contains a two-dimensional fixed component, then by Lemma \ref{squares} $\int_{M_{\max}} = c_{1}(N_{\max})^2=1$, therefore applying Proposition \ref{zerofourcasetwo}, $\int_{\Sigma}c_{1}(M)=6$, where $\Sigma$ is the two dimensional component, which contradictions the fact that $\iota(M)=4$.

If $X \cong \mathbb{P}^4$,  then by Lemma \ref{squares} $(d_{1},d_2) = (2,4),(0,6)$.  If $(d_1,d_2) = (0,6)$ by Lemma \ref{zerosixcase} the action is FP equivalent to   Example \ref{projectiveone}. If $(d_1,d_2) = (2,4)$ by Proposition \ref{twofourcase} the action is FP equivalent to   Example \ref{exampletwofour}.

If $X $ is in the family $X^m_{8}$ then by Lemma \ref{squares} $(d_1,d_2) = (0,4),(4,4).$ The invariants of  $X^m_{8}$ are as follows, by  \cite[Lemma 4.4]{PZ1}   $b_{4}(X^{m}_{8}) = 8$, by Equation (\ref{degreegenus}) $\int_{X^{m}_{8}} c_1^4 = 224$.  Assuming $(d_1,d_2) = (0,4),$ if there is an extremal component of dimension $2$, then $\iota(X)=3$ by Lemma \ref{sphereindex} which is impossible. If there is no fixed component of dimension $2$, then by Lemma \ref{zerofourcaseone}, Lemma \ref{intone}, Lemma \ref{inttwo},  $\int_{M} c_{1}^4 = 240 + 176  - 16 b_{4}(M) = 288$ which is a contradiction. Therefore, $(d_1,d_2) = (4,4)$ and by Proposition \ref{fourfourcase} the action is FP-equivalent to Example \ref{kuznetsov}.
\end{proof}

Departamento di Mathematica e Informatica ``U. Dini", Universit\`{a} Di  Firenze.\\

\noindent nicholas.lindsay@unifi.it.

\begin{thebibliography}{9}
\bibitem{CK3} I. Charton and L. Kessler. Monotone symplectic six-manifolds that admit a hamiltonian gkm action
are diffeomorphic to smooth fano threefolds. Transformation Groups. 10.1007/s00031-025-09914-2.
\end{thebibliography}
\end{document}